\documentclass[11pt]{amsart}

\usepackage[centertags]{amsmath}
\usepackage{amsfonts}
\usepackage{amssymb}
\usepackage{amsthm}
\usepackage[english]{babel}
\usepackage[active]{srcltx}
\usepackage[pdftex]{graphicx} 
\DeclareGraphicsExtensions{.png,.pdf,.jpg}

\newcommand{\dist}{\operatorname{dist}}

\newcommand{\Real}{\mathbb{R}}

\newcommand{\Lip}{\operatorname{Lip}}
\newcommand{\supp}{\operatorname{supp}}

\newtheorem{thm}{Theorem}[section]
\newtheorem{cor}[thm]{Corollary}
\newtheorem{lem}[thm]{Lemma}
\newtheorem{prop}[thm]{Proposition}

\newtheorem{rem}[thm]{Remark}

\newtheorem{defn}[thm]{Definition}

\numberwithin{equation}{section}

\begin{document}

\title[Approximation and Extension on Banach manifolds]{On some problems on smooth approximation and smooth extension
of Lipschitz functions on Banach-Finsler Manifolds}

\author{M. Jim{\'e}nez-Sevilla and L. S\'anchez-Gonz\'alez}

\address{Departamento de An{\'a}lisis Matem{\'a}tico\\ Facultad de
Matem{\'a}ticas\\ Universidad Complutense\\ 28040 Madrid, Spain}

\thanks{Supported in part by DGES (Spain) Project MTM2009-07848. L. S\'anchez-Gonz\'alez has  also been supported by grant MEC AP2007-00868}

\email{marjim@mat.ucm.es,
lfsanche@mat.ucm.es}



\date{December, 2010}

\maketitle

\begin{abstract}
Let us consider a Riemannian manifold $M$ (either separable or non-separable). We prove  that, for  every $\varepsilon>0$,
every  Lipschitz function  $f:M\rightarrow \mathbb R$ can be uniformly approximated  by a Lipschitz, $C^1$-smooth function $g$ with $\Lip(g)\le  \Lip(f)+\varepsilon$.  As a consequence, every Riemannian  manifold is uniformly bumpable. These  results  extend  to the non-separable setting those given  in \cite{AzFeMeRa} for separable Riemannian manifolds.
The results are presented in the context of $C^\ell$ Finsler manifolds modeled on Banach spaces. Sufficient conditions are given
on the Finsler manifold $M$ (and the Banach space $X$ where $M$  is modeled), so that every Lipschitz function  $f:M\rightarrow \mathbb R$ can be uniformly approximated by a Lipschitz, $C^k$-smooth function $g$ with $\Lip(g)\le C \Lip(f)$ (for some $C$ depending only on $X$). Some applications of these results are also given as well as a characterization, on the separable case, of the class of $C^\ell$ Finsler manifolds satisfying the above property of approximation. Finally, we  give  sufficient conditions on the  $C^1$ Finsler manifold $M$ and $X$, to ensure the existence of Lipschitz and $C^1$-smooth extensions
of every real-valued function $f$ defined on a submanifold $N$ of  $M$ provided $f$ is $C^1$-smooth on $N$ and Lipschitz with the
metric induced by $M$. 
\end{abstract}


\section{Introduction}
In this work we address the problem whether every Lipschitz function $f:M\rightarrow \mathbb R$ defined on a {\em non-separable} Riemannian
manifold  can be uniformly approximated by  a Lipschitz, $C^\infty$-smooth function $g:M\rightarrow \mathbb R$.
The study of this problem was motivated  by the work  in \cite{AzFeMeRa},
where this result of approximation is stated for   {\em separable}
Riemannian manifolds. The question whether this result holds for every Riemannian manifold is posed in  \cite{AzFeMe}, \cite{AzFeMeRa} and \cite{GaJaRa}.
 A positive answer to this question provides nice applications such as: (i)  the uniformly bumpable character of every Riemannian manifold,  (ii) Deville-Godefroy-Zizler smooth variational principle  holds for every complete  Riemannian manifold \cite{AzFeMe} and (iii) the infinite-dimensional version of the Myers-Nakai theorem given in \cite{GaJaRa}   holds for every complete Riemannian manifold
 (either separable or non separable) as well.

The problem of the  uniform approximation of  Lipschitz functions defined
on a Banach space by Lipschitz and $C^1$-smooth functions  has been largely studied. J. M. Lasry and P. L. Lions used sup-inf convolution techniques to answer positively to this problem on every Hilbert space  \cite{LL}. R. Fry introduced
``sup-partitions of unity'', a key tool  to answer positively to this problem for the case of bounded and  Lipschitz functions defined on
 Banach spaces with separable dual \cite{Fry1}. Later on,  D. Azagra, R. Fry and V. Montesinos extended this result in \cite{azafrymon}. In particular,
 they obtained $C^k$ smoothness of the Lipschitz approximating functions whenever  $X$ is separable and admits a Lipschitz $C^{k}$-smooth bump function.
 Recently, it has been shown that this kind of approximation holds for every real-valued Lipschitz function defined on  a Banach space with separable dual \cite{HajekJohanis} (see also  \cite[Lemma 1]{Azafrykeener}),    on $c_0(\Gamma)$ (for every non-empty set of indexes $\Gamma$) \cite{HajJohc0} and on a larger class of non-separable Banach spaces \cite{HajekJohanis}. Moreover, they obtain an upper bound (that only depends on $X$) of the ratio between the Lipschitz constant of the constructed smooth functions that uniformly approximate to a function $f$ and the Lipschitz constant of $f$.

 We study the problem in the context of smooth Banach-Finsler manifolds. Our aim is to study sufficient conditions on a smooth
 Banach-Finsler manifold $M$  so that  the above result on uniform approximation of  Lipschitz functions defined on $M$ holds. We consider the setting of Banach-Finsler manifolds so that we can obtain a unified approach to this problem for both
 Riemannian and non-Riemannian manifolds, such as any submanifold of a Banach space $X$, that is a natural continuation of the study done in infinite dimensional Banach spaces.

  In Section 2 we recall  the concepts of $C^\ell$  Finsler manifold in the sense of Palais and Neeb-Upmeier introduced in  \cite{Palais}, \cite{Neeb}, \cite{Upmeier}  (both modeled on a Banach space $X$ either separable or non separable).  We introduce  the notions of
  $C^\ell$  Finsler manifold in the sense of Neeb-Upmeier {\em weak-uniform} and {\em uniform}  as a generalization of  a $C^\ell$  Finsler manifold in the sense of Palais and a Riemannian manifold, respectively.
Some results related to mean value inequalities are provided for  $C^\ell$  Finsler manifolds  in the sense of  Palais,  Neeb-Upmeier   {\em weak-uniform} and Neeb-Upmeier {\em uniform}.
 A result on the existence of suitable local bi-Lipschitz diffeomorphisms is also given as an essential tool to establish the results on approximation and extension in the next sections.

 In Section 3, we prove  that every real-valued and Lipschitz function  defined on a $C^\ell$ Finsler manifold $M$ (in the sense of
  Neeb-Upmeier) {\em weak-uniform} modeled on a Banach space $X$ can be uniformly approximated by a $C^k$-smooth and Lipschitz function provided the Banach space $X$ satisfies a similar approximation property (which we shall denote throughout this work ($*^k$)) in a {\em uniform way}. In other words,  if for every Lipschitz function $f: X\rightarrow \mathbb R$ and every $\varepsilon>0$, there is a  $C^k$-smooth and Lipschitz function $g$ such that  $|f(x)-g(x)|<\varepsilon$ for every $x\in M$ and moreover $\Lip(g)\le C \Lip(f)$, where  the constant $C$ only depends on $X$ and  $C$ {\em does not depend on a certain class of  equivalent norms} considered in $X$. This class of norms
is closely related to the set of  norms defined in the tangent spaces to $M$. We shall prove that, for $\ell=1$, the above assertion holds whenever the manifold $M$ is modeled on a Banach space $X$ with separable dual or  $M$ is a (separable or non-separable) Riemannian manifold.
In the proof of this assertion, we use the results given in the previous section as well as the existence of  smooth and Lipschitz partitions of unity subordinated to suitable open covers of  the manifold $M$ (see \cite{Rudin}, \cite{HajekJohanis} and \cite{MarLuis}) and the ideas of the separable Riemannian case \cite{AzFeMeRa}. A similar result is provided in the case that $M$
is a $C^\ell$ Finsler manifold
(in the sense of Neeb-Upmeier) {\em uniform}. It is worth mentioning  that, in this case, the constant $C$ is not required to be independent of a certain class of norms  considered in $X$.

 In Section 4 several applications are given. Under the above assumptions on the manifold $M$ (in particular, if $M$ is a Riemannian manifold), it can be  deduced that $M$ is {\em uniformly bumpable}  (this concept was first defined in \cite{AzFeMe})  and thus the Deville-Godefroy-Zizler smooth variational principle  holds whenever $M$ is complete. This generalizes the result given in \cite{AzFeMe} for separable and complete Riemannian manifolds. Moreover, it can be deduced that the infinite-dimensional version of the Myers-Nakai theorem for separable Riemannian manifolds given in \cite{GaJaRa} holds for every infinite-dimensional complete Riemannian manifold. An interesting open problem related to this result is whether an infinite-dimensional version of the Myers-Nakai theorem can be obtained for a certain class of (infinite-dimensional) complete Finsler manifolds.
 For a study of a finite dimensional version of the Myers-Nakai theorem for certain classes of Finsler manifolds see  \cite{Rangel-Tesis}.

 In Section 5, we follow the ideas of \cite{Fry1}, \cite{azafrymon} and \cite{HajekJohanis} to establish a characterization of the class of separable $C^\ell$ Finsler manifolds $M$ in the sense of Neeb-Upmeier which are  {\em $C^k$-smooth uniformly bumpable}
 (see Definition \ref{bumpable})  as those having the property that
 every   Lipschitz  function $f$ defined on $M$ can be uniformly approximated by a  $C^k$-smooth and Lipschitz
 function $g$ such that $\Lip(g)\le C \Lip(f)$ and $C$ only depends on $M$.

 In Section 6 the following extension result is established on a $C^\ell$ Finsler manifold $M$ (in the sense of Neeb-Upmeier) {\em weak-uniform} modeled on a Banach space $X$: for every $C^k$-smooth and real-valued function $f$ defined on a closed submanifold $N$ of $M$, such that $f$ is Lipschitz (with respect to the metric of the manifold $M$),  there is a $C^k$-smooth and Lipschitz extension of $f$ defined on $M$,  provided  the Banach space $X$ satisfies the approximation property ($*^k$) in a {\em uniform way}. The proof relies on a related result established in \cite{Azafrykeener} for Banach spaces with separable dual and in  \cite{MarLuis} for a larger class of Banach spaces.
\smallskip

The notation we use is standard. The norm in a Banach space $X$ is denoted by $||\cdot||$.  The open ball with center $x\in X$ and radius $r>0$ is denoted by $B(x,r)$. A $C^k$-smooth bump function $b:X\to \Real$ is a $C^k$-smooth function on X with bounded, non-empty support, where $\supp(b) =\overline{\{x \in  X : b(x)\neq 0\}}$.
If $M$ is a Banach-Finsler manifold, we denote by $T_xM$ the tangent space of $M$ at $x$. Recall that the tangent bundle of $M$ is $TM=\{(x,v):x\in M \text{ and } v\in T_x M\}$. We refer to \cite{Deville}, \cite{fabianhajek}, \cite{Lang}, \cite{Deim} and \cite{Upmeier}  for additional definitions.


\section{Preliminaries and Tools}



Let us begin with  the introduction  of the class of manifolds we will consider in this work.

\begin{defn} \label{defFinsler}
Let $M$ be a (paracompact) $C^\ell$ Banach manifold modeled on a Banach space $(X,||\cdot||)$. Let us denote by  $TM$ the tangent  bundle
of $M$ and consider a continuous  map $||\cdot||_M: TM\to [0,\infty)$. We say that

\begin{itemize}

\item[{\bf(F1)}]  $(M,||\cdot||_M)$ is a \textbf{$C^\ell$ Finsler manifold in the sense of Palais} (see \cite{Palais},  \cite{Deim}, \cite{Rabier})
if $||\cdot||_M$ satisfies the following conditions:
\begin{enumerate}
\item[(P1)] For every $x\in M$, the map $||\cdot||_x:={||\cdot||_M}_{\mid_{T_xM}}:T_xM\to [0,\infty)$ is a norm on the tangent space $T_xM$ such that for every chart $\varphi:U\to X$ with $x \in U$, the norm $v\in X \mapsto ||d\varphi^{-1}(\varphi(x))(v)||_x$ is equivalent to $||\cdot||$ on $X$.
\item[(P2)] For every $x_0\in M$, $\varepsilon>0$ and every chart $\varphi:U\to X$ with $x_0\in U$, there is an open neighborhood $W$ of $x_0$  such that if  $x\in W$ and $v\in X$, then
\begin{equation*}\label{palaisdef}
\frac{1}{1+\varepsilon}||d\varphi^{-1}(\varphi(x_0))(v)||_{x_0}\le ||d\varphi^{-1}(\varphi(x))(v)||_{x}\le (1+\varepsilon)||d\varphi^{-1}(\varphi(x_0))(v)||_{x_0}.
\end{equation*}
In terms of equivalence of norms, the above inequalities yield to the fact that the norms $||d\varphi^{-1}(\varphi(x))(\cdot)||_{x}$,
are $(1+\varepsilon)$-equivalent to $||d\varphi^{-1}(\varphi(x_0))(\cdot)||_{x_0}$.
\end{enumerate}



\item[{\bf (F2)}]  $(M,||\cdot||_M)$ is a \textbf{$C^\ell$ Finsler manifold in the sense of Neeb-Upmeier} (\cite{Neeb}; Upmeier in \cite{Upmeier} denotes these manifolds by  {\bf {\em normed Banach manifolds}}) if $||\cdot||_M$  satisfies   conditions {\em (P1)} and
\begin{enumerate}
\item[(NU1)] for every $x_0\in M$ there exists a chart $\varphi:U\to X$ with $x_0\in U$ and $K_{x_0}\ge 1$ such that for every $x\in U$ and every $v\in T_x M$,
\begin{equation}\label{ineq-Neeb-Upmeier}
\frac{1}{K_{x_0}}||v||_x\le ||d\varphi(x)(v)||\le K_{x_0}||v||_x.
\end{equation}
 Equivalently,  $(M,||\cdot||_M)$ is a $C^\ell$ Finsler manifold in the sense of Neeb-Upmeier if it satisfies conditions {\em (P1)} and
\item[(NU2)]  for every $x_0\in M$ there exists a  chart $\varphi:U\to X$ with $x_0\in U$
and a constant $M_{x_0}\ge 1$ such that for every $x\in U$ and every $v\in X$,
 \begin{equation*}
\frac{1}{M_{x_0}}||d\varphi^{-1}(\varphi(x_0))(v)||_{x_0}\le ||d\varphi^{-1}(\varphi(x))(v)||_{x}\le M_{x_0}||d\varphi^{-1}(\varphi(x_0))(v)||_{x_0}.
\end{equation*}
\end{enumerate}

\item[{\bf (F3)}] $(M,||\cdot||_M)$   is a \textbf{ $C^\ell$ Finsler manifold  in the sense of  Neeb-Upmeier   weak-uniform} if it satisfies {\em (P1)} and there is $K\ge 1$ such that
\begin{enumerate}
\item[(NU3)]  for every $x_0\in M$, there exists a chart $\varphi:U\to X$ with $x_0\in U$ satisfying, for every $x\in U$ and $v\in  X$,
\begin{equation} \label{ineq-K-weak-uniform}
\frac{1}{K}||d\varphi^{-1}(\varphi(x_0))(v)||_{x_0}\le ||d\varphi^{-1}(\varphi(x))(v)||_{x}\le K||d\varphi^{-1}(\varphi(x_0))(v)||_{x_0}.
\end{equation}
\end{enumerate}
In this case, we will say that $(M,|| \cdot ||_M)$ is \textbf{ $K$-weak-uniform}.

\item[{\bf (F4)}]  $(M,||\cdot||_M)$ is a  \textbf{ $C^\ell$ Finsler manifold in the sense of  Neeb-Upmeier uniform} if $||\cdot||_M$  satisfies {\em (P1)} and
\begin{enumerate}
\item[(NU4)] there is $S\ge 1$ such that for each $x_0\in M$ there exists a chart $\varphi:U\to X$ with $x_0\in U$ and
\begin{equation}\label{ineq-S-uniform}
\frac{1}{S}||v||_x\le ||d\varphi(x)(v)||\le S||v||_x, \quad {\text whenever \ } x\in U
\text{ and } v\in T_x M.
\end{equation}
\end{enumerate}
In this case, we will say that $(M,|| \cdot ||_M)$ is  \textbf{ $S$-uniform}.
\end{itemize}
\end{defn}


\begin{rem}\label{remark}
\begin{enumerate}
\item Clearly, {\em (F1)} $\Rightarrow$ {\em (F3)}. Also,  {\em (F4)} $\Rightarrow$ {\em
(F3)} $\Rightarrow$ {\em (F2)} (see Figure 1).
\item Every Riemannian manifold is  a $C^\infty$ Finsler manifold in the sense of Palais (see  \cite{Palais}) and a $C^\infty$ Finsler manifold in the sense of Neeb-Upmeier uniform.
\item The concepts of $C^\ell$ Finsler manifold in the sense of Palais and $C^\ell$ Finsler manifold in the sense of Neeb-Upmeier  are equivalent for finite-dimensional manifolds.
\item Nevertheless, in the infinite-dimensional setting, there are examples of  $C^\ell$ Finsler manifolds in the sense of Neeb-Upmeier that do not satisfy the Palais condition {\em (P2)}  (see \cite[Example 10]{GaGuJa}).
\item \label {remark-k} Note that if $M$ is a $C^\ell$ Finsler manifold in the sense of Palais, then it is $K$-weak-uniform for every $K>1$.
Also, if $M$ is $S$-uniform, then $M$ is $S^2$-weak-uniform.
Indeed, inequality \eqref{ineq-S-uniform} is equivalent to the fact that  $||d\varphi^{-1}(\varphi(x))(\cdot)||_x$
is $S$-equivalent to $||\cdot||$ on $X$, for every $x\in U$. Now, it can be easily checked that this implies that
$||d\varphi^{-1}(\varphi(x))(\cdot)||_x$ is $S^2$-equivalent to $||d\varphi^{-1}(\varphi(x_0))(\cdot)||_{x_0}$
for every $x\in U$.
\item \label{remark-R} It can be checked that  condition (NU3) in the definition of $C^\ell$ Finsler manifold in the sense of Neeb-Upmeier weak-uniform yields to the following condition:
\begin{enumerate}
\item[(NU3')] for every $R>K$, $x_0\in M$ and every chart $\varphi:U\to X$ with $x_0\in U$, there is an open subset $W$ with $x_0\in W\subset U$ such that
\begin{equation} \label{ineq-K-weak-uniform'}
\frac{1}{R}||d\varphi^{-1}(\varphi(x_0))(v)||_{x_0}\le ||d\varphi^{-1}(\varphi(x))(v)||_{x}\le R||d\varphi^{-1}(\varphi(x_0))(v)||_{x_0},
\end{equation}
whenever $x\in W$ and $v\in X$,
\end{enumerate}
\end{enumerate}
\end{rem}



\begin{center}
\begin{figure}[h]
{\includegraphics[width=9.34cm]{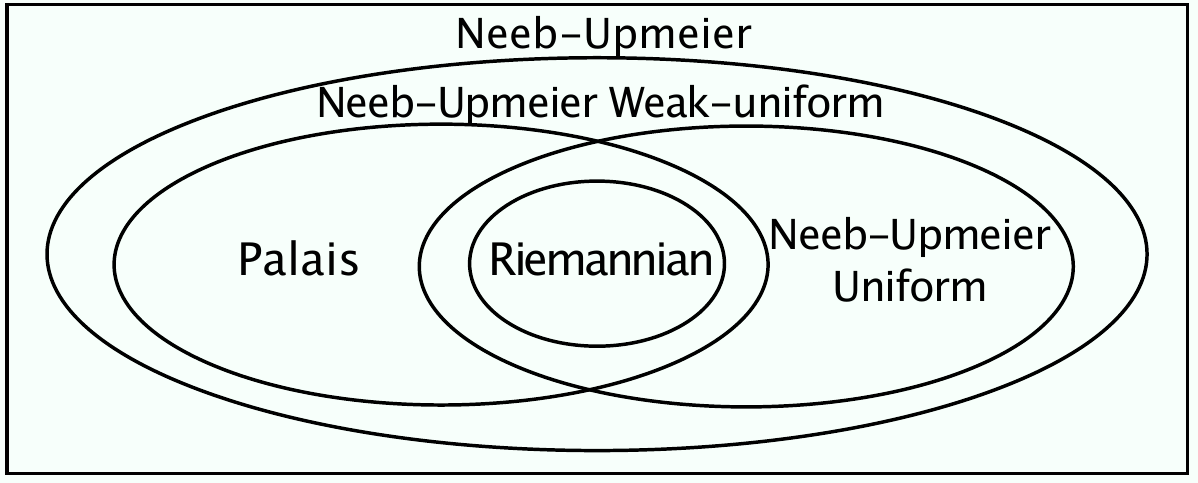}}
\label{cusp}
\caption{{\small Finsler Manifolds}}
\label{figura}
\end{figure}
\end{center}


Let $M$ be a Banach manifold and $f:M\to\Real$ a differentiable function at $p\in M$. The norm of $df(p)\in {T_p M}^*$ is given by
\begin{equation*}
||df(p)||_p= \sup\{|df(p)(v)|:v\in T_p M, ||v||_p\le 1\}.
\end{equation*}
Let us consider a differentiable function  $f:M\to N$  between Banach manifolds $M$ and $N$. The norm of the derivative at the point $p\in M$ is defined as
\begin{align*}
& ||df(p)||_p= \sup\{||df(p)(v)||_{f(p)}:v\in T_p M, ||v||_p\le 1\} = &\\
& \qquad \qquad =\sup\{\xi(df(p)(v)):\xi\in {T_{f(p)} N}^*,\ v\in T_pM\ \text{and}\ ||v||_{p}=1=||\xi||^*_{f(p)} \}.&
\end{align*}
Recall that if $(M, ||\cdot||_M)$ is a Finsler manifold in the sense of Neeb-Upmeier, the {\em length} of a piecewise $C^1$ smooth
path $c:[a,b]\rightarrow M$ is defined as $\ell(c):=\int_{a}^b||c'(t)||_{c(t)}\,dt$. Besides, if $M$ is connected, then it is
connected by piecewise $C^1$ smooth paths, and the associated
{\em Finsler metric} $d_M$ on $M$ is defined as
\begin{equation*}
d_M(p,q)=\inf\{\ell(c): \, c \text{ is a piecewise } C^1 \text{ smooth path connecting } p \text{ to } q\}.
\end{equation*}
Recall that the Finsler metric is consistent with the topology given in $M$ (see \cite{Palais}, \cite[Proposition 12.22]{Upmeier}). The open ball of center $p\in M$ and radius
$r>0$ is  denoted by $B_M(p,r):=\{q\in M:\, d_M(p,q)<r\}$. The Lipschitz constant $\Lip(f)$ of a Lipschitz function $f:M\rightarrow N$, where
$M$ and $N$ are Finsler manifolds, is defined as $\Lip(f)=\sup\{\frac{d_N(f(x),f(y))}{d_M(x,y)}: x,y \in M, x\not=y\}$.

In the following proposition we obtain some  ``mean value" inequalities. The ideas of the proof  follow those of the Riemannian case (see \cite{AzFeMe}).

\begin{prop}(Mean value inequalities). \label{mean:value}
Let  $M$ and $N$ be  $C^1$ Finsler manifolds in the sense of Neeb-Upmeier,  and  $f:M\to N$ be a  $C^1$-smooth function.
 \begin{itemize}
 \item[(i)]  If  $\sup\{||df(x)||_x:x\in M\}<\infty$, then $f$ is Lipschitz and  $\Lip(f)\le \sup\{||df(x)||_x:x\in M\}$.
 \item[(ii)] If $f$ is Lipschitz, the manifold $M$ is $K$-weak-uniform and the manifold $N$ is $P$-weak-uniform, then $\sup\{||df(x)||_x:x\in M\}\le KP \Lip(f)$.
 \item[(ii')] If $f$ is Lipschitz, the manifold $M$ is $K$-uniform and the manifold $N$ is $P$-uniform, then $\sup\{||df(x)||_x:x\in M\}\le K^2P^2 \Lip(f)$.
\item[(iii)] Thus, if  $f$ is Lipschitz and the manifolds $M$ and $N$ are Finsler manifolds in the sense of Palais,
 then $\sup\{||df(x)||_x:x\in M\}=\Lip(f)$.
\end{itemize}
\end{prop}

\begin{proof}
$(i)$
Let us consider $p,q\in M$ with $d_M(p,q)<\infty$, and $\varepsilon>0$. Then there is   a piecewise $C^1$  smooth path  $\gamma:[0,T]\to M$ joining $p$ and $q$ with $\ell(\gamma)\le d_M(p,q)+\varepsilon/C$, where $C=\sup\{||df(x)||_x:x\in M\}$. In order to simplify the proof, let us assume that $\gamma$ is $C^1$
smooth (the general case follows straightforward).
Now, we define $\beta:[0,T]\to N$ as $\beta(t)=f(\gamma(t))$. Then,
 $\beta$ joins the points $f(p)$ and $f(q)$, and
\begin{eqnarray*}
d_N(f(p),f(q))&\le & \ell(\beta)\le \int_0^T ||df(\gamma(t))(\gamma'(t))||_{\beta(t)}dt \le
\int_0^T ||df(\gamma(t))||_{\gamma(t)}||\gamma'(t)||_{\gamma(t)}\,dt
  \\ &\le &C\int_0^T ||\gamma'(t)||_{\gamma(t)}\,dt \le Cd_M(p,q)+\varepsilon.
\end{eqnarray*}
Thus, $d_N(f(p),f(q))\le Cd_M(p,q)+\varepsilon$ for every $\varepsilon>0$.
Then,  $d_N(f(p),f(q))\le Cd_M(p,q)$ and  $\Lip(f)\le C$.


$(ii)$ First, let us consider the case $N=\Real$. Let us denote by $L:=\Lip(f)$. Let us take $x_0\in M$ and $\varphi:U\to X$ a chart with $x_0\in U$, $\varphi(x_0)=0$ and satisfying inequality \eqref{ineq-K-weak-uniform}. In order to simplify the proof, let us denote by  $|||\cdot|||$
 the norm $||d\varphi^{-1}(0)(\cdot)||_{x_0}$ on $X$.  Let us take a ball  $B(0,r)\subset \varphi(U)$ and check that the function $f\circ \varphi^{-1}:B(0,r)\to \Real$ is $KL$-Lipschitz with respect to the norm $|||\cdot|||$. Indeed,
 for every pair of points $x,y\in B(0,r)$, we define $\gamma:[0,1]\to M$ as $\gamma(t):=\varphi^{-1}(ty+(1-t)x) \in U$.
 By inequality \eqref{ineq-K-weak-uniform} we obtain
\begin{align*}
&|f\circ\varphi^{-1} (x)-f\circ\varphi^{-1}(y)|\le Ld_M(\varphi^{-1}(x),\varphi^{-1}(y))\le L \ell(\gamma) = &\\
& \qquad \qquad = L\int_0^1 ||\gamma'(t)||_{\gamma(t)}dt =  L \int_0^1 ||d\varphi^{-1}(ty+(1-t)x)(y-x)||_{\gamma(t)}dt \le &\\
& \qquad \qquad \le KL  \int_0^1 ||d\varphi^{-1}(0)(y-x)||_{x_0}dt= KL |||x-y|||.&
\end{align*}
Thus, from the mean value inequalities in Banach spaces we obtain  $|d (f\circ \varphi^{-1})(0)(v)|\le KL |||v|||$ for every
$v\in T_{x_0}M$. Therefore,
\begin{align*}
&|df(x_0)(v)|=|d(f\circ \varphi^{-1}\circ \varphi )(x_0)(v)| =|d (f\circ \varphi^{-1})(0)(d\varphi(x_0)(v))|\le &\\
& \qquad \qquad \quad \le KL|||d\varphi(x_0)(v)||| = KL||d\varphi^{-1}(0)(d\varphi(x_0)(v))||_{x_0}=KL||v||_{x_0},&
\end{align*}
for every
$v\in T_{x_0}M$ and $||df(x_0)||_{x_0}\le KL$. Since this conclusion holds  for every $x_0\in M$, we deduce that
\begin{equation*}
\sup_{x\in M} ||df(x)||_x\le KL.
\end{equation*}

Now, let us consider the general case, i.e.  $f:M\to N$ where $M$ is a weak-uniform Neeb-Upmeier
manifold with constant $K\ge 1$ and  $N$ is a weak-uniform Neeb-Upmeier manifold with constant $P\ge 1$. If there is a point $x_0\in M$ such that $||df(x_0)||_{x_0}>KPL$, then there are $\xi \in T_{f(x_0)}N^*$ and $v \in T_{x_0} M$ such that $||v||_{x_0}=||\xi||^*_{f(x_0)}=1$, and $|\xi(df(x_0)(v))|>KPL$.  Let us take  a chart of $N$ at $f(x_0)$ satisfying inequality \eqref{ineq-K-weak-uniform} with constant $P$, which we shall denote by $\psi:V\to Y$, where $N$ is modeled on the Banach space $Y$. Also, let us take $r>0$ such that $f(x_0)\in B_N(f(x_0),r/4)\subset B_N(f(x_0),r)\subset V$ and define the function
\begin{eqnarray*}
&&g: f^{-1}(B_N(f(x_0),r/4))\to \Real,
\\
&&g(x)=\xi\circ d\psi^{-1}(\psi(f(x_0)))(\psi(f(x))).
\end{eqnarray*}
Then, on the one hand,
\begin{align}\label{equation:derivadag}
|d g(x_0)(v)|=|\xi\circ d\psi^{-1}(\psi(f(x_0)))(d\psi(f(x_0))df(x_0)(v))|=&\\
=|\xi(d f(x_0)(v))|>KPL.&\nonumber
\end{align}
On the other hand, let us check that the function $g$ is  $PL$-Lipschitz with the distance $d_M$. Indeed, first let us show
 that $\psi$ is $P$-Lipschitz with the norm $|||\cdot|||:=|| d\psi^{-1}(\psi(f(x_0)))(\cdot)||_{f(x_0)}$. Since  $\psi :V\to Y$ satisfies
 inequality \eqref{ineq-K-weak-uniform} with constant $P$,  we have for every $z\in V$ and $v\in T_z N$,
\begin{align*}
&|||d\psi(z)(v)|||=||d\psi^{-1}(\psi(f(x_0)))(d\psi(z)(v))||_{f(x_0)}\le &\\
& \qquad \qquad \qquad \le  P ||d\psi^{-1}(\psi(z))(d\psi(z)(v))||_{z} = P||v||_z &
\end{align*}
Hence,  $|||d\psi(z)|||:=\sup\{|||d\psi(z)(v)|||: v\in T_z N\ \text{and}\ ||v||_z\le 1\}\le P$ for every $z\in V$, and thus
$\sup \{|||d\psi(z)|||: z\in V\}\le P$.
Also, let us check that,   for every $z,z'\in B_N(f(x_0),r/4)$,
\begin{equation*}d_N(z,z')=\inf\{\ell(\gamma): \gamma \text{ is a piecewise } C^1\text{ path connecting } z
\text{ and } z'  \text{ with } \gamma \subset V \}. \end{equation*}
 Indeed, if there are $z,z' \in B_N(f(x_0),r/4)$ with $d_N(z,z')=\inf\{\ell(\gamma): \gamma$ is a piecewise $C^1$ path connecting $z$ and $z' \}< \inf\{\ell(\gamma): \gamma$ is a piecewise $C^1$ path connecting $z$ and $z'$ with $\gamma \subset V \}$, then there is a path $\gamma:[0,1]\to N$ such that $\gamma(0)=z$,  $\gamma(1)=z'$, $\ell(\gamma)<d_N(z,z')+r/4$ and $\gamma(t)\notin V$ for some $t\in [0,1]$. Then $d_N(f(x_0),\gamma(t))\ge r$ and $d_N(z,\gamma(t))\le \ell(\gamma)\le d_N(z,z')+r/4$. This yields,
\begin{equation*}
r\le d_N(f(x_0),\gamma(t))\le d_N(f(x_0),z)+d_N(z,\gamma(t))\le d_N(f(x_0),z)+d_N(z,z')+r/4< r,
\end{equation*}
which is a contradiction.
Now, we can follow the proof of part  (i) to deduce that $\psi$ is $P$-Lipschitz on $B_N(f(x_0),r/4)$ with the norm $|||\cdot|||
=|| d\psi^{-1}(\psi(f(x_0)))(\cdot)||_{f(x_0)}$ on $Y$.
Finally, for every $x,y\in f^{-1}(B(f(x_0),r/4))$, we have
\begin{align}\label{equation:Lipschitzg}
|g(x)-g(y)|=|\xi\circ d\psi^{-1}(\psi(f(x_0)))(\psi\circ f(x)-\psi\circ f(y))|\le \nonumber& \\
  \le ||d\psi^{-1}(\psi(f(x_0)))(\psi\circ f(x)-\psi\circ f(y))||_{f(x_0)} = &\\
 \qquad  =  |||\psi\circ f(x)-\psi\circ f(y)||| \le P d_N(f(x),f(y))\le PL d_M(x,y).& \nonumber
\end{align}
Then $g:f^{-1}(B(f(x_0),r/4))\to \Real$ is $PL$-Lipschitz, and by the real case we have that  $\sup\{||dg(x)||_x:x\in f^{-1}(B(f(x_0),r/4))\}\le KPL$, which contradicts (\ref{equation:derivadag}).
\smallskip

(ii') and (iii) follow from (ii) and Remark \ref{remark}(\ref{remark-k}).\end{proof}


The following lemma provides a local bi-Lipschitz behaviour of the charts of a
$C^1$ Finsler manifold.

\begin{lem}\label{desigualdades:BiLipschitz}
Let us consider a  $C^1$ Finsler manifold $M$.
\begin{enumerate} \item If  $M$ is $K$-weak-uniform, then for every $x_0\in M$ and every chart
$(U,\varphi)$  with $x_0\in U$ satisfying inequality \eqref{ineq-K-weak-uniform}, there exists an open neighborhood
 $V\subset U$ of $x_0$ satisfying
\begin{equation}\label{B-L1}
\frac{1}{K}d_M(p,q)\le |||\varphi(p)-\varphi(q)|||\le K d_M(p,q), \quad \text{ for every } p,q\in V,
\end{equation}
where   $|||\cdot|||$ is the (equivalent) norm $||d\varphi^{-1}(\varphi(x_0))(\cdot)||_{x_0}$ defined  on $X$.
\item If the manifold $M$ is  $K$-uniform, then   for every $x_0\in M$ and every chart
$(U,\varphi)$ with $x_0\in U$ satisfying condition \eqref{ineq-S-uniform},  there exists an open neighborhood
 $V\subset U$ of $x_0$ satisfying
\begin{equation}\label{B-L2}
\frac{1}{K}d_M(p,q)\le ||\varphi(p)-\varphi(q)||\le K d_M(p,q), \quad \text{ for every } p,q\in V.
\end{equation}
\end{enumerate}
\end{lem}
\begin{proof}
Let us assume the hypothesis in (1) holds. The arguments given in the proof of  Proposition \ref{mean:value} yield to the existence
of $r>0$ with  $B_M(x_0,r/4)  \subset  B_M(x_0,r) \subset U\subset M$ such that if $p,q\in B_M(x_0,r/4)$, then
 $$d_M(p,q)=\inf\{\ell(\gamma): \gamma \text{ is a piecewise }
 C^1 \text{ path connecting } p \text{ and } q  \text{ with } \gamma \subset U \}.$$
  Let us consider $p,q\in B_M(x_0,r/4)$,  $\varepsilon>0$ and a piecewise $C^1$  smooth path  $\gamma:[0,T]\to U$ joining $p$
  and $q$ with $\ell(\gamma)\le d_M(p,q)+\varepsilon/K$.
 Let us define $\beta:[0,T]\to X$ as $\beta(t)=\varphi(\gamma(t))$. Then,
 $\beta$ joins the points $\varphi(p)$ and $\varphi(q)$, and from inequality \eqref{ineq-K-weak-uniform} we obtain
\begin{align*}
|||\varphi(p)-\varphi(q)|||\le
\ell(\beta)\le \int_0^T |||d\varphi(\gamma(t))(\gamma'(t))|||\,dt \le K \int_0^T ||\gamma'(t)||_{\gamma(t)}dt = &\\
=K\ell(\gamma) \le Kd_M(p,q)+\varepsilon.&
\end{align*}
Now, let us consider $\varphi^{-1}$ and $s>0$ such that $B(\varphi(x_0),s)\subset \varphi(B_M(x_0,r/4))$. For $x,y\in B(\varphi(x_0),s)$, let us define the path $\gamma:[0,1]\to M$ as
$\gamma(t):=\varphi^{-1}(ty+(1-t)x) \in B_M(x_0,r/4).$
Then,
\begin{align*}
d_M(\varphi^{-1}(x),\varphi^{-1}(y)) &\le \ell(\gamma) =\int_0^1 ||\gamma'(t)||_{\gamma(t)}dt =
 \int_0^1 ||d\varphi^{-1}(\varphi(\gamma(t)))(y-x)||_{\gamma(t)}dt \\ &\le  \int_0^1 K ||d\varphi^{-1}(\varphi(x_0))(y-x)||_{x_0}
 = K |||x-y|||.\end{align*}
Finally, let us define the open set $V:=\varphi^{-1}(B(\varphi(x_0),s))$.

 The proof under the hypothesis given in  (2) follows along the same lines.
  \end{proof}


\section{Smooth Approximation of Functions}\label{3}

Before stating the main result of this section, let us define property $(*^k)$ for Banach spaces as the following Lipschitz and $\mathcal{C}^k$-smooth approximation property for Lipschitz mappings on Banach spaces.
\begin{defn}
A Banach space $(X, ||\cdot||)$ satisfies  property $(*^k)$  if there is a constant $C_0\ge1$, which only  depends on the space $(X,||\cdot||)$, such that, for any Lipschitz function $f:X\to \Real$ and any $\varepsilon>0$ there is a Lipschitz, $C^k$-smooth function $K:X\to \Real$ such that
\begin{equation*}
|f(x)-K(x)|<\varepsilon \text{ for all } x\in X   \text{ and } \Lip(K)\leq C_0 \Lip(f).
\end{equation*}
\end{defn}
Notice that if a Banach space $X$ satisfies  property  $(*^k)$, then for every Lipschitz function $f: A\to \Real$
(where  $A$ is a subset of  $X$)   and every $\varepsilon>0$ there is a Lipschitz, $C^k$-smooth function $K:X\to \Real$ such that
\begin{equation*}
\left\vert f(x)-K(x) \right\vert<\varepsilon \text{ for all } x\in A  \text{ and } \Lip(K)\leq C_0 \Lip(f).
\end{equation*}
Indeed, there exists a Lipschitz extension $F:X\to\Real$ of $f$ such that $\Lip(F)=\Lip(f)$  (for instance  $x\mapsto  \inf_{y\in A} \{f(y)+\Lip(f)||x-y||\}$), and applying property $(*^k)$ to $F$  the assertion is obtained.

\begin{rem}
\begin{enumerate}
\item Every finite-dimensional Banach space $X$ admits  property $(*^\infty)$. Since the  functions $K(\cdot)$ are constructed by means of convolutions, it can be easily checked that the constant $C_0$ can be taken as $1$ for every equivalent
norm $||\cdot||$ considered in $X$.
\item Every Hilbert space $H$ admits property $(*^1)$  (see \cite{LL}). Also, from the construction of the
functions $K(\cdot)$ with inf-sup-convolution formulas, it can be  easily checked that the constant
$C_0$ can be taken as $1$ for every Hilbertian norm $||\cdot||$ considered in $H$.
\item Every separable Banach space with a $C^k$-smooth and Lipschitz bump function satisfies property $(*^k)$ (see \cite{azafrymon}, \cite{Azafrykeener}, \cite{Fry1} and \cite{HajekJohanis}).
Moreover, the constant $C_0$ can be obtained to be independent of the equivalent norm considered in $X$. Indeed, a careful examination of the proofs given in these  papers, allows us to ensure that $X$ satisfies property $(*^1)$ with  constant $C_0\le 600$  for every  $C^1$-smooth norm defined on $X$.
Now, using  the density of the set of $C^1$-smooth norms on the metric space $(\mathcal N(X), h)$ of all (equivalent) norms
defined in $X$ with the Hausdorff metric $h$ (see \cite[Theorem II.4.1]{Deville}),
it can be shown that $X$ satisfies property $(*^1)$ with $C_0\le 601$ for any equivalent norm on $X$. Now, we can deduce property  $(*^k)$ with constant $C_0\le 602$ (independently of the equivalent norm considered on $X$)
from the results on $C^1$-fine approximation of $C^1$-smooth functions by
$C^k$-smooth functions \cite{Azfrygiljarlovo}, \cite{HajekJohanis} and \cite{Moulis}.
\item A Banach space $X$ such that there is a bi-Lipschitz homeomorphism between $X$ and a subset of $c_0(\Gamma)$, for some set $\Gamma\neq \emptyset$, whose coordinate functions are $\mathcal{C}^k$-smooth, satisfies property $(*^k)$ (see  \cite{HajekJohanis}). Unfortunately, we do not know if, in this general case, the constant $C_0$ can be obtained to be independent
of the (equivalent) norm considered in $X$.
\end{enumerate}
\end{rem}


The following lemma is quite useful in  approximation of functions on Banach spaces. It  provides the existence of suitable open coverings on a (paracompact) Banach manifold, which will be key to obtain results on smooth approximations and smooth extensions.
Let us recall that  the distance between two sets $A$ and $B$ of a metric space is
 $\dist(A,B)=\inf \{d(a,b): a\in A, \ b\in B\}$.

\begin{lem}(See M.E. Rudin, \cite{Rudin}) \label{Rudin}
Let $E$ be a metric space, $\mathcal{U}=\{U_r\}_{r\in\Omega}$ be an open cover of $E$. Then, there are open
refinements $\{V_{n,r}\}_{n\in\mathbb{N},r\in\Omega}$ and $\{W_{n,r}\}_{n\in\mathbb{N},r\in\Omega}$ of $\mathcal{U}$ satisfying  the following properties:
\begin{itemize}
\item[(i)] $V_{n,r}\subset W_{n,r}\subset U_r$ for all $n\in\mathbb{N}$ and $r\in \Omega$,
\item[(ii)] $\dist(V_{n,r},E\setminus W_{n,r})\geq  1/2^{n+1}$ for all $n\in\mathbb{N}$ and $r\in\Omega$,
\item[(iii)] $\dist(W_{n,r},W_{n,r'})\geq 1/2^{n+1}$ for any $n\in\mathbb{N}$ and $r,r'\in\Omega$, $r\not=r'$,
\item[(iv)] for every $x\in E$ there is an open ball $B(x,s_x)$ of $E$ and a natural number $n_x$ such that
     \begin{enumerate}
         \item[(a)] if $i>n_x$, then $B(x,s_x)\cap W_{i,r}=\emptyset$ for every $r\in\Omega$,
         \item[(b)] if $i\leq n_x$, then $B(x,s_x)\cap W_{i,r}\neq\emptyset$ for at most one $r\in\Omega$.
     \end{enumerate}
\end{itemize}
\end{lem}


In the following, we will extend to a certain class of Finsler manifolds, the result on approximation of Lipschitz functions by smooth and Lipschitz  functions defined on separable Riemannian manifolds given in  \cite{AzFeMeRa}. This result is new, even in the case when $M$ is a non-separable Riemannian manifold.

\begin{thm}\label{PalaisFinslerTheorem}
Let $M$ be a $C^\ell$ Finsler $K$-weak-uniform manifold modeled on a Banach space $X$ which admits property $(*^k)$ and the constant $C_0$ does not depend on the (equivalent) norm. For every Lipschitz function $f:M\to\Real$, and any  continuous function $\varepsilon:M\to(0,\infty)$  there is a Lipschitz, $C^m$-smooth
 function $g:M\to\Real$ ($m:={\min\{\ell,k\}}$) such that
\begin{equation*}
\left\vert g(p)-f(p) \right\vert<\varepsilon(p), \  \  || dg(p)||_p\le C_1 \Lip(f) \quad  \text{ for every } p\in M,
\end{equation*}
and therefore, $ \Lip(g)\le C_1 \Lip(f)$, where $C_1:=2C_0 K^2$.
\end{thm}
\begin{proof}
We can assume that $L:=\Lip(f)>0$ and
 $0<\varepsilon(p)<C_0K^2L$, for all $p\in M$.
 For every $p\in M$, there is $\delta_p>0$ such that $\varepsilon(p)/3<\varepsilon(q)$ for every $q\in B_M(p,3\delta_p)$ and a $C^\ell$-smooth chart $\varphi_p:B_M(p,3\delta_p)\to X$ with $\varphi_p(p)=0$, satisfying (P1), inequality \eqref{ineq-K-weak-uniform} and inequality \eqref{B-L1} for all point of the ball $B_M(p,3\delta_p)$.
In particular, $\varphi_p$ and $\varphi_p^{-1}$ are Lipschitz with the (equivalent) norm $||d\varphi^{-1}_{p}(0)(\cdot)||_{p}$ considered on $X$, $\Lip(\varphi_p) \le K$ and $\Lip(\varphi_p^{-1}) \le K$.

Let us  consider an  open cover  $\bigcup_{\gamma\in \Gamma} B_M(p_\gamma,\delta_\gamma)$ of $M$,
where $\delta_{\gamma}:=\delta_{p_\gamma}$ for some set of indexes $ \Gamma$.  Also, let us
write $\varphi_\gamma:=\varphi_{p_\gamma}$, $\varepsilon_\gamma:=\varepsilon(p_\gamma)$ and
$|||\cdot|||_\gamma:=||d\varphi^{-1}_{\gamma}(0)(\cdot)||_{p_\gamma}$.
Let us define, for every $\gamma\in \Gamma$,
\begin{equation*}f_\gamma: \varphi_\gamma(B_M(p_\gamma,3\delta_\gamma))\subset X \to \Real,  \qquad
f_\gamma(x):=f(\varphi_{\gamma}^{-1}(x)),\end{equation*}
 which is $KL$-Lipschitz with the norm $|||\cdot|||_\gamma$.
By Lemma \ref{Rudin}, there are open refinements $\{V_{n,\gamma}\}_{n\in\mathbb{N},\gamma\in\Gamma}$ and $\{W_{n,\gamma}\}_{n\in\mathbb{N},\gamma\in\Gamma}$ of $\{B_M(p_\gamma,2\delta_\gamma)\}_{\gamma\in\Gamma}$ satisfying  properties $(i)-(iv)$ of Lemma \ref{Rudin}.


Now, we need the following lemma related to the existence of smooth and Lipschitz  partitions of unity on a manifold $M$.
First, let us recall  the definition of a smooth and Lipschitz partitions of unity.


\begin{defn} A collection of real-valued, $C^k$-smooth and Lipschitz functions $\{\psi_i\}_{i\in I}$  defined on a Finsler manifold $M$ is a $C^k$-smooth and Lipschitz partition of unity  subordinated to the open cover $\mathcal{U}=\{U_r\}_{r\in \Omega}$ of $M$ whether
 (1)   $\psi_i\ge 0$ on $M$ for every $i\in I$, (2) the family $\{\supp (\psi_i)\}_{i\in I}$ is locally finite, where $\supp (\psi_i) =
 \overline{\{x\in M: \psi_i(x)\neq 0\}}$, i. e. for every $x\in M$ there is an open neighborhood $U$ of $x$ and a finite subset $J\subset I$ such that
 $\supp (\psi_i)\cap U=\emptyset$ for every  $i\in I\setminus J$,
 (3) for
 every $i\in I$ there is $r\in \Omega$ such that $\supp (\psi_i)\subset  U_{r}$, and (4) $\sum_{i\in I} \psi_i(x)=1$ for every
 $x\in M$.
\end{defn}


\begin{lem}\label{partition1}
Under the assumptions of Theorem \ref{PalaisFinslerTheorem}, there is a $C^m$-smooth partition of unity of $M$  $\{\psi_{n,\gamma}\}_{n\in\mathbb{N}, \gamma\in\Gamma}$
such that $\supp(\psi_{n,\gamma}) \subset W_{n,\gamma}$
and $\psi_{n,\gamma}$ is Lipschitz for every
$n\in \mathbb N$, $\gamma \in \Gamma$. In fact,   $||d\psi_{n,\gamma}(p)||_p
\le n15 C_0K^22^{n+1}$ for all $p\in M$, and thus $\Lip(\psi_{n,\gamma})\le  n15 C_0K^22^{n+1}$ for all  $n\in \mathbb N$ and
$\gamma \in \Gamma$.
\end{lem}

Let us assume that  Lemma \ref{partition1} has been proved.  Let us denote by $L_{n,\gamma}:=
\max\{1,\sup\{||d\psi_{n,\gamma}(p)||_p: \, {p\in M}\}\}$. Since $X$ admits property $(*^k)$ and the constant $C_0$ does not depend
on the equivalent norm considered on $X$,  there is a $C^k$-smooth, Lipschitz function
$g_{n,\gamma}:X\to \Real$ such that
\begin{equation*}
|g_{n,\gamma}(x)-{f}_{\gamma}(x)|\le \frac{\varepsilon_\gamma/3}{2^{n+2} L_{n,\gamma}} \ \text{ for all }\ x\in
 \varphi_\gamma(B_M(p_\gamma,3\delta_\gamma))
\end{equation*}
and $\Lip(g_{n,\gamma})\le C_0 \Lip({f}_{\gamma})\le  C_0 K L$ with the norm $|||\cdot|||_\gamma$ on $X$.
Let us define the function $g:M\to \Real$ as
\begin{equation*}
g(p):=\sum_{n\in\mathbb{N}, \gamma\in \Gamma} \psi_{n,\gamma}(p)g_{n,\gamma}(\varphi_{\gamma}(p)),
\quad  p\in M.
\end{equation*}
Now, if $p\not\in B_M(p_\gamma, 2\delta_\gamma)$, then $\psi_{n,\gamma}(p)=0$ and $\psi_{n,\gamma}(p)g_{n,\gamma}(\varphi_{\gamma}(p))=0$. Since  $\supp (\psi_{n,\gamma})\subset W_{n,\gamma}\subset B_M(p_\gamma,2\delta_\gamma)$,
it is clear that $p\mapsto \psi_{n,\gamma}(p)g_{n,\gamma}(\varphi_{\gamma}(p))$ is $C^m$-smooth on $M$, for each $n\in\mathbb{N}$ and $\gamma\in \Gamma$.
 Moreover, $\{\supp (\psi_{n,\gamma})\}_{n\in \mathbb N, \gamma \in \Gamma}$ is locally finite,
    and thus $g$ is well defined and  $C^m$-smooth on $M$. \newline
Note that, if $\psi_{n,\gamma}(p)\neq 0$, then $p\in \supp (\psi_{n,\gamma})\subset B_M(p_\gamma,2\delta_\gamma)$ and thus $f(p)={f}_\gamma(\varphi_\gamma(p))$. Hence,
\begin{align*}
 &|g(p)-f(p)|=&\\
 & =|\sum_{n\in\mathbb{N},\gamma\in \Gamma} \psi_{n,\gamma}(p)g_{n,\gamma}(\varphi_{\gamma}(p))-f(p)|= |\sum_{n\in\mathbb{N},\gamma\in \Gamma} \psi_{n,\gamma}(p)(g_{n,\gamma}(\varphi_{\gamma}(p))-f(p))|  =&\\
&= |\sum_{\{(n,\gamma) : \psi_{n,\gamma}(p)\neq0 \}} \psi_{n,\gamma}(p)(g_{n,\gamma}(\varphi_{\gamma}(p))-{f}_\gamma(\varphi_{\gamma}(p)))| \le  &\\
& \le \sum_{\{(n,\gamma) : \psi_{n,\gamma}(p)\neq0 \}} \psi_{n,\gamma}(p)  \frac{\varepsilon_\gamma/3}{2^{n+2}  L_{n,\gamma}}<\varepsilon(p).&
\end{align*}
Let us check that $g$ is  $2C_0 K^2 L$-Lipschitz on $M$.
Recall that $\sum_{\mathbb N\times \Gamma}\psi_{n,\gamma}(p)= 1$  for all $p\in M$, and thus $\sum_{\mathbb N\times \Gamma}d\psi_{n,\gamma}(p)=0$ for all $p\in M$. Also, properties (i) and (ii) of the open refinement $\{W_{n,\gamma}\}_{n\in \mathbb N, \gamma \in \Gamma}$ imply that for every $p\in M$  and $n\in\mathbb{N}$,
there is at most one $\gamma\in\Gamma$, which we shall denote by $\gamma_p(n)$,  such that $p\in \supp  (\psi_{n,\gamma})$. Let us  define the finite set $F_p:=\{(n, \gamma)\in \mathbb N \times \Gamma:
 p\in \supp(\psi_{n,\gamma})\}=\{(n, \gamma_p(n))\in \mathbb N \times \Gamma:
 p\in \supp(\psi_{n,\gamma_p(n)})\}$.
Recall that, if we consider the norm $|||\cdot |||_\gamma$ on $X$, then $\Lip(g_{n,\gamma})\le  C_0 K L $ and thus
$|||dg_{n,\gamma}(x)|||:=\sup\{ |dg_{n,\gamma}(x)(v)| : |||v|||_\gamma\le 1\} \le  C_0 K L$ for all $x\in X$.
Also, $|||d\varphi_{\gamma}(p)|||:=\sup\{|||d\varphi_{\gamma}(p)(v)|||_\gamma: ||v||_p\le 1\}\le K$ whenever $p\in B(p_\gamma, 3\delta_{\gamma})$.
Therefore, we obtain that $||d(g_{n,\gamma}\circ\varphi_{\gamma})(p)||_p
\le C_0K^2L$ whenever $p \in  B(p_\gamma, 3\delta_{\gamma})$ and
\begin{align*}
&||dg(p)||_p = ||\sum_{(n,\gamma)\in F_p} g_{n,\gamma}(\varphi_{\gamma}(p)) d\psi_{n,\gamma}(p)+
\sum_{(n,\gamma)\in F_p}\psi_{n,\gamma}(p) d(g_{n,\gamma}\circ\varphi_{\gamma})(p)||_p=\\
&=||\sum_{ (n,\gamma)\in F_p} (g_{n,\gamma}(\varphi_{\gamma}(p))-f(p)) d\psi_{n,\gamma}(p)+
\sum_{(n,\gamma)\in F_p}\psi_{n,\gamma}(p) d(g_{n,\gamma}\circ\varphi_{\gamma})(p)||_p\le \\
& \le  \sum_{(n,\gamma) \in F_p} |g_{n,\gamma}(\varphi_{\gamma}(p))-{f}_\gamma(\varphi_{\gamma}(p))| \,||d \psi_{n,\gamma}(p)||_p  + \sum_{(n,\gamma)\in F_p}\psi_{n,\gamma}(p)C_0 K^2L   \le \\
 &  \le \sum_{\{n:\,(n,\gamma_p(n))\in F_p\}} \frac{\varepsilon(p)}{2^{n+2} L_{n,\gamma_p(n)}}  L_{n,\gamma_p(n)}+C_0
K^2L \le \\
 & \le \varepsilon(p)/4+C_0K^2L<2C_0K^2L.
\end{align*}
Finally, by Proposition \ref{mean:value}(i), $\Lip(g) \le \sup\{||dg(p)||_p : p\in M\}\le 2C_0 K^2 L$ which finishes the proof of Theorem \ref{PalaisFinslerTheorem}.


\medskip


Now, let us prove Lemma \ref{partition1}.
Let us consider the two refinements $\{V_{n,\gamma}\}_{n\in\mathbb{N},\gamma\in\Gamma}$ and $\{W_{n,\gamma}\}_{n\in\mathbb{N},\gamma\in\Gamma}$ of $\{B_M(p_\gamma,2\delta_\gamma)\}_{\gamma\in\Gamma}$ satisfying the properties $(i)-(iv)$ of Lemma \ref{Rudin}. Recall that $\dist_M(V_{n,\gamma},M\setminus W_{n,\gamma})\ge {1}/{2^{n+1}}$ and
$\dist_M(W_{n,\gamma},W_{n,\gamma'})\geq 1/2^{n+1}$ for  every $\gamma,\,\gamma'\in\Gamma$, $\gamma\not=\gamma'$, and  every $n\in\mathbb{N}$. Also, recall that $\varphi_{\gamma}:B_M(p_\gamma,3\delta_\gamma)\to \varphi_\gamma(B_M({p_\gamma},3\delta_\gamma)) := \widetilde{B}_\gamma \subset X$ satisfies
\begin{equation*}
\frac{1}{K} d_M(p,q)\le |||\varphi_{\gamma}(p)-\varphi_{\gamma}(q)|||_\gamma\le K d_M(p,q), \quad \text{ for } p,q\in B_M(p_\gamma,3\delta_\gamma).
\end{equation*}
Let us denote $\widetilde{V}_{n,\gamma}:=\varphi_{\gamma}(V_{n,\gamma})$ and $\widetilde{W}_{n,\gamma}:=\varphi_{\gamma}(W_{n,\gamma})$. Clearly,  $\widetilde{V}_{n,\gamma}\subset \widetilde{W}_{n,\gamma}\subset \varphi_\gamma(B_M({p_\gamma},3\delta_\gamma))=\widetilde{B}_\gamma  \subset X$. Also, for every $x\in  \widetilde{V}_{n,\gamma}\subset X$ and $y\in \widetilde{B}_\gamma \setminus  \widetilde{W}_{n,\gamma}\subset X$, there are $p\in V_{n,\gamma}$ and $q\in B_M(p_\gamma,3\delta_\gamma)\setminus W_{n,\gamma}$ such that $\varphi_{\gamma}(p)=x$ and $\varphi_{\gamma}(q)=y$. Thus, $|||x-y|||_\gamma=|||\varphi_{\gamma}(p)-\varphi_{\gamma}(q)|||_\gamma  \ge \frac{1}{K}  d_M(p,q)\ge \frac{1}{K2^{n+1}}$.
Let us define $\dist_\gamma(A,B):=\inf\{|||x-y|||_\gamma: x\in A \ \text{and}\ y\in B\}$ for any pair of subsets $A,B\subset X$.
Then,   $\dist_\gamma(\widetilde{V}_{n,\gamma},\widetilde{B}_\gamma \setminus \widetilde{W}_{n,\gamma})\ge \frac{1}{K2^{n+1}}$.

Let us define $\phi_{n,\gamma}:X\to \Real$ as $\phi_{n,\gamma}(x)=\dist_\gamma (x, \widetilde{V}_{n,\gamma})$. Then,
 $\phi_{n,\gamma}(\widetilde{V}_{n,\gamma})=0$ and $\inf \phi_{n,\gamma}(\widetilde{B}_\gamma \setminus  \widetilde{W}_{n,\gamma})\ge \frac{1}{K 2^{n+1}}$. Let us take a  Lipschitz function $\theta_n:\Real\to [0,1]$ such that
$\theta_n(t)=1$ for $t<\frac{1}{4 K2^{n+1}}$, $\theta_n(t)=0$  for
$t>\frac{1}{2K2^{n+1}}$
with $\Lip(\theta_n)\le 5K2^{n+1}$.
Then, $(\theta_n\circ \phi_{n,\gamma})(\widetilde{V}_{n,\gamma})=1$,  $(\theta_n\circ \phi_{n,\gamma})(\widetilde{B}_\gamma \setminus \widetilde{W}_{n,\gamma})=0$ and $\Lip(\theta_n\circ\phi_{n,\gamma})\le 5 K2^{n+1}$ (with the norm $|||\cdot|||_\gamma$).

Now, by property $(*^k)$, we can find  $C^k$-smooth and Lipschitz functions
 $\xi_{n,\gamma}: X\to \Real$ such that
 \begin{equation*}
 \sup_{y\in X}\{|\xi_{n,\gamma}(y)-(\theta_n\circ \phi_{n,\gamma})(y)|\}<1/4 \quad \text{ and } \quad
 \Lip(\xi_{n,\gamma})\le C_0\Lip(\theta_n\circ\phi_{n,\gamma}),
 \end{equation*}
  with the norm $|||\cdot|||_\gamma$, for every $\gamma\in\Gamma$ and $n\in\mathbb{N}$.

Let us take a $C^\infty$-smooth Lipschitz function $\theta:\Real\to[0,1]$ such that $\theta(t)=0$ whenever $t<\frac{1}{4}$,  $\theta(t)=1$ whenever $t>\frac{3}{4}$ and  $\Lip(\theta)\le 3$.
Let us define $\widetilde{h}_{n,\gamma}:X\to [0,1]$ as $\widetilde{h}_{n,\gamma}(x)=\theta(\xi_{n,\gamma}(x))$, for every $n\in\mathbb{N}$ and $\gamma\in\Gamma$.
Then, $\widetilde{h}_{n,\gamma}(x)$ is $C^k$-smooth,  $\Lip(\widetilde{h}_{n,\gamma})\le 15 C_0K2^{n+1} $ (with the norm $|||\cdot|||_\gamma$), $\widetilde{h}_{n,\gamma}(\widetilde{V}_{n,\gamma})=1$ and $\widetilde{h}_{n,\gamma}(\widetilde{B}_\gamma \setminus \widetilde{W}_{n,\gamma})=0$.

Now, let us define $h_{n,\gamma}:M\to [0,1]$ as
\begin{equation*}
h_{n,\gamma}(p)=
\begin{cases}
\widetilde{h}_{n,\gamma}(\varphi_{\gamma}(p)) & \text{if $p\in B_M(p_\gamma,3\delta_\gamma)$},\\
0 & \text{otherwise}.
\end{cases}
\end{equation*}
Then, the function $h_{n,\gamma}$ is $C^m$-smooth, $\supp(h_{n,\gamma})\subset W_{n,\gamma} \subset B_M(p_\gamma,2\delta_\gamma)$, $||dh_{n,\gamma}(p)||_p\le 15 C_0K^22^{n+1}$ for every $p\in M$ and  thus $\Lip(h_{n,\gamma})\le 15 C_0K^22^{n+1}$.

Let us  define $h_n: M\rightarrow \mathbb R$ as $h_n(p)=\sum_{\gamma\in\Gamma} h_{n,\gamma}(p)$, for every $n\in\mathbb{N}$.  Since $\dist(W_{n,\gamma},W_{n,\gamma'})>0$ whenever $\gamma\neq \gamma'$, we deduce that $h_n$ is $C^m$-smooth. Also, $h_n(\bigcup_{\gamma\in\Gamma} V_{n,\gamma})=1$ and $h_n(M\setminus \bigcup_{\gamma\in \Gamma} W_{n,\gamma})=0$. In addition, $||dh_n(p)||_p\le 15 C_0K^22^{n+1}$ for every $p\in M$ and thus  $\Lip(h_n)\le  15 C_0K^22^{n+1}$.
Finally, let us define
\begin{equation*}
\psi_{1,\gamma}=h_{1,\gamma} \quad \text{ and }  \quad  \psi_{n,\gamma}=h_{n,\gamma}(1-h_1)\cdots(1-h_{n-1}), \quad
 \text{for }  n\ge 2.
 \end{equation*}
Clearly the functions $\{ \psi_{n,\gamma}\}_{n\in \mathbb N, \gamma \in \Gamma}$ are  $C^m$-smooth functions,
$||d\psi_{n,\gamma}(p)||_p \le  n15 C_0K^22^{n+1}$ for all $p \in M$  (and thus
$\Lip(\psi_{n,\gamma}) \le  n15 C_0K^22^{n+1}$),
$\supp (\psi_{n,\gamma})\subset \supp (h_{n,\gamma})\subset W_{n,\gamma}\subset B_M(p_\gamma,2\delta_\gamma)$.
In addition, for every $p\in M$,
\begin{align*}
\sum_{n\in\mathbb{N},\gamma\in\Gamma} \psi_{n,\gamma}(p)&=\sum_{\gamma\in\Gamma} \psi_{1,\gamma}(p)+
\sum_{n\ge 2}\left(\sum_{\gamma\in\Gamma}h_{n,\gamma}(p)\right)\prod_{i=1}^{n-1}(1-h_i(p))=\\
&= h_1(p)+\sum_{n\ge2} h_n(p)\prod_{i=1}^{n-1}(1-h_i(p))=1.
\end{align*}
Hence,  $\{\psi_{n,\gamma}\}_{n\in\mathbb{N},\gamma\in\Gamma}$ is a $C^m$-smooth  partition of unity subordinated to the open cover $\{W_{n,\gamma}\}_{n\in\mathbb{N},\gamma\in\Gamma}$ of $M$ with  $||d\psi_{n,\gamma}(p)||_p
\le n15 C_0K^22^{n+1}$ and $\Lip(\psi_{n,\gamma})\le  n15 C_0K^22^{n+1}$  for all $p\in M$,  $n\in \mathbb N$ and
$\gamma \in \Gamma$. This finishes the proof of Lemma \ref{partition1}. \end{proof}

If we do not assume that the constant $C_0$ is independent of the (equivalent) norm considered in the Banach space $X$, a  similar result to Theorem \ref{PalaisFinslerTheorem} can be obtained for smooth Finsler manifolds in the sense of Neeb-Upmeier $K$-uniform  modeled on a Banach space $X$.


\begin{thm}\label{FinslerTheorem}
Let $M$ be a $C^\ell$ Finsler manifold in the sense of Neeb-Upmeier $K$-uniform, modeled on a Banach space $(X,||\cdot||)$ which admits property $(*^k)$. For every Lipschitz function $f:M\to\Real$, any  continuous function $\varepsilon:M\to(0,\infty)$  there is a Lipschitz, $C^m$-smooth function $g:M\to\Real$  ($m:={\min\{\ell,k\}}$)  such that
\begin{equation*}
|g(p)-f(p)|<\varepsilon(p),   \  \  || dg(p)||_p\le C_1 \Lip(f) \quad  \text{ for every } p\in M,
\end{equation*}
and thus $ \Lip(g)\le C_1 \Lip(f)$, where $C_1:=2C_0 K^2$
and $C_0$ is the constant given by property $(*^k)$.
\end{thm}
 The proof of Theorem \ref{FinslerTheorem} follows along the same lines as that for Theorem
\ref{PalaisFinslerTheorem}. Let us  indicate that,
in this case,  throughout the proof the norm considered in $X$  is
 $||\cdot||$ (instead of $|||\cdot|||_\gamma$).

\section{Corollaries}
In this section we will give several corollaries of Theorem \ref{PalaisFinslerTheorem}.
The first collorary is related to the concept of {\em uniformly bumpable Banach manifold}.
Let us recall that the existence of  smooth and Lipschitz bump functions on a Banach space is an essential tool to  obtain approximation of Lipschitz functions by Lipschitz and smooth functions defined on Banach spaces (see \cite{azafrymon, Fry1, HajekJohanis}). A generalization of this concept to manifolds is the notion of  \emph{uniformly bumpable manifold}, which was introduced by Azagra, Ferrera and L\'opez-Mesas \cite{AzFeMe} for Riemannian manifolds. A natural extension to every  \emph{Finsler manifold} can be defined in the same way, as follows.
\begin{defn}\label{bumpable}
A  $C^\ell$  Finsler manifold $M$ in the sense of Neeb-Upmeier  is \textbf{$C^k$-uniformly bumpable} (with $k\le \ell$) whenever there are $R>1$ and $r>0$ such that for every $p\in M$ and $\delta\in (0,r)$ there exists a $C^k$-smooth function $b:M\to[0,1]$ such that:
\begin{enumerate}
\item $b(p)=1$,
\item $b(q)=0$ whenever $d_M(p,q)\ge \delta$,
\item $\sup_{q\in M} ||d b(q)||_q \le R/\delta$.
\end{enumerate}
\end{defn}

Note that this is not a restrictive definition. In fact,  Azagra,   Ferrera,   L\'opez-Mesas and  Rangel \cite{AzFeMeRa} proved that  every separable Riemannian manifold is $C^\infty$-uniformly bumpable. Now, we can show that a rich class of Finsler manifolds, which includes  every  Riemannian manifold (separable or non-separable),  is uniformly bumpable. This result answers a problem posed in \cite{AzFeMe, AzFeMeRa, GaJaRa}.

\begin{cor}\label{approximation:uniformly}
Let $M$ be a  $C^\ell$ Finsler manifold  in the sense of  Neeb-Upmeier satisfying one of the following conditions:
\begin{enumerate}
\item $M$ is $K$-weak-uniform and it is  modeled on a Banach space $X$ which admits property $(*^k)$ and the constant $C_0$ does not depend on the norm.
\item $M$ is $K$-uniform and it is modeled on a Banach space $X$ which admits property $(*^k)$.
\end{enumerate}
Then, $M$ is $C^m$-uniformly bumpable  with $m:=\min\{\ell,k\}$.
\end{cor}
\begin{proof}
The  assertion follows from Theorem \ref{PalaisFinslerTheorem} and Theorem \ref{FinslerTheorem}
 in a similar way to the Riemannian case.
Let us give the proof under the assumption given in (1)  for completeness. For every $r>0$,  $0<\delta< r$ and $p\in M$, let us define the function $f:M\to [0,1]$ such that
\begin{equation*}
f(q)=
\begin{cases}
1-\frac{d_M(q,p)}{\delta}Ê& \text{if $d_M(q,p)\le \delta$,}\\
0 & \text{if $d_M(q,p)\ge \delta$}.\\
\end{cases}
\end{equation*}
The function $f$ is $\frac{1}{\delta}$-Lipschitz, $f(p)=1$ and $f(q)=0$ whenever $q\not\in B_M(p,\delta)$.
Let us define $R:=3KC_1$   (where $C_1=2C_0K^2$ is the constant given in  Theorem \ref{PalaisFinslerTheorem})
and take $\varepsilon:=\frac{1}{4}$. By Theorem \ref{PalaisFinslerTheorem}, there is a $C^m$-smooth function $g:M\to\Real$ such that  $\sup_{p\in M}||dg(p)||_p\le C_1 \Lip(f)$,  thus  $\Lip(g)\le C_1 \Lip(f) =\frac{C_1}{\delta}$ and $|g(q)-f(q)|<\frac{1}{4}$ for every $q\in M$.
Let us take a suitable $C^\infty$-smooth and Lipschitz function $\theta:\Real\to[0,1]$  such that $\theta(t)=0$ whenever $t\le \frac{1}{4}$ and $\theta(t)=1$ for $t\ge \frac{3}{4}$ (with $\Lip(\theta)\le 3$). Let us  define $b(q)=\theta(g(q))$ for $q\in M$. It is
clear that  $b$ is $C^m$-smooth, $\sup_{p\in M}||db(p)||_p\le \frac{3C_1}{\delta}$, and thus
$\Lip(b)\le \frac{3C_1}{\delta}$,  $b(p)=1$ and $b(q)=0$ for $q\not\in B_M(p,\delta)$. Finally, we define
 $R:= 3C_1=6C_0K^2$ and this finishes the proof.
\end{proof}


The results given so far provide some interesting consequences on Riemannian manifolds.
The following corollary provides a generalization (in the  $C^1$-smoothness case) to the  non-separable setting
of the result given in \cite{AzFeMeRa}  for separable   Riemannian manifolds.
\begin{cor}
Let $M$ be a Riemannian manifold. Then, for every Lipschitz function $f:M\to\Real$, every continuous function  $\varepsilon: M\to (0,\infty)$ and $r>0$ there is a $C^1$-smooth and Lipschitz function $g:M\to\Real$ such that $|g(p)-f(p)|<\varepsilon(p)$ for every
$p\in M$ and $\Lip(g)\le \Lip(f) +r$.
\end{cor}
Let us notice that for separable Riemannian manifolds,  the  Lipschitz function $g$ that approximates $f$ can be obtained to be
$C^\infty$-smooth (see \cite{AzFeMeRa}). Unfortunately, in the non-separable case we  can only ensure that $g$ is $C^1$-smooth.
\begin{cor} \label{bumpRiemann}
Every Riemannian manifold is $C^1$-uniformly bumpable.
\end{cor}
Let $M$ be a Riemannian manifold. Let us denote by  $C_b^1(M)$ the algebra of all bounded, Lipschitz and $C^1$-smooth functions $f:M\to\Real$.   It is easy to check  that  $C_b^1(M)$ is a Banach space endowed with the norm
$||f||_{C_b^1(M)}:=\max\{||f||_\infty,||df||_\infty\}$ (where $||f||_\infty=\sup_{p\in M}|f(p)|$ and $||df||_\infty=\sup_{p\in M}||df(p)||_p$). Moreover,  it is a {\em Banach algebra} with the norm $2||\cdot||_{C_b^1(M)}$ (see \cite{GaJaRa}). Recall that two {\em normed algebras} $(A,||\cdot||_A)$ and $(B,||\cdot||_B)$   are said to be {\em equivalent as normed algebras} whenever there exists an {\em algebra isomorphism} $T:A\to B$  such that $||T(a)||_B=||a||_A$ for every $a\in A$. Also, the Riemannian manifolds $M$ and $N$ are said to be {\em equivalent} whenever there is a {\em Riemannian isometry}
 $h:M\to N$, i.e. $h$ is  a $C^1$-diffeomorphism from $M$ onto $N$ satisfying
\begin{equation*}
\langle dh(x)(v),dh(x)(w)\rangle _{h(x)}=\langle v,w \rangle _x
\end{equation*}
for every $x\in M$ and every $v,wÊ\in T_x M$ (where $\langle \cdot, \cdot \rangle_p$ is the escalar product defined in $T_pM$).
Garrido,  Jaramillo and  Rangel proved  in \cite{GaJaRa}  a version of the Myers-Nakai theorem (see \cite{Myers}, \cite{Nakai})
for infinite-dimensional Riemannian manifolds under the assumption that the Riemannian manifold is $C^1$-uniformly bumpable.
Therefore, from  \cite{GaJaRa} and Corollary \ref{bumpRiemann}, we can deduce the following
assertion.
\begin{cor}
Let $M$ and $N$ be complete Riemannian manifolds. Then $M$ and $N$ are equivalent Riemannian manifolds if, and only if, $C_b^1(M)$ and $C_b^1(N)$ are equivalent as normed algebras. Moreover, every normed algebra isomorphism $T:C_b^1(N)\to C_b^1(M)$ is of the form $T(f)=f\circ h$, where $h:M\to N$ is a Riemannian isometry.
\end{cor}

A version  for uniformly bumpable complete Riemannian manifolds of the {\em Deville-Godefroy-Zizler smooth variational principle}
\cite{Deville} (DGZ smooth variational principle, for short) was proved  in \cite{AzFeMe}. Thus, from \cite{AzFeMe} and Corollary \ref{bumpRiemann} we deduce  the following corollary.
Recall that a function $f:M \to \Real\cup\{\infty\}$ {\em attains its strong minimum} on $M$ at $x\in M$ if $f(x)=\inf\{f(z):z\in M\}$ and $d_M(x_n,x)\to 0$ whenever $\{x_n\}_{n=1}^\infty$ is a sequence of points of $M$ such that  $f(x_n)\to f(x)$. A function
$f:M\to \Real\cup\{\infty\}$ is said to be proper whether $f\neq \infty$.

\begin{cor}(DGZ smooth variational principle for Riemannian manifolds).
Let $M$ be a complete Riemannian manifold  and let $f:M\to \Real\cup\{\infty\}$ be a lower semicontinuous (lsc) function which is bounded below and proper. Then, for each $\varepsilon>0$ there is a bounded $C^1$-smooth and Lipschitz function $\varphi:M\to\Real$ such that
\begin{enumerate}
\item $f-\varphi$ attains its strong minimum on $M$,
\item $||\varphi||_\infty <\varepsilon$ and $||d\varphi||_\infty<\varepsilon$.
\end{enumerate}
\end{cor}


Actually,  following  the proof given for Riemannian manifolds  \cite{AzFeMe}, we can extend this variational principle to the class of  $C^1$  Finsler manifolds in the sense of Neeb-Upmeier weak-uniform, provided they are $C^1$-uniformly bumpable. Let us indicate that, as in the previous Sections 2 and 3, the use of the norms $|||\cdot|||_\gamma$ are required  in order to prove the result.
\begin{cor} \label{DGZfinslermanifolds} (DGZ smooth variational principle for Finsler manifolds).
Let $M$ be a complete and $C^1$-uniformly bumpable $C^1$ Finsler manifold in the sense of Neeb-Upmeier  $K$-weak-uniform, and let $f:M\to \Real\cup\{\infty\}$ be a lsc function which is bounded below and proper. Then, for each $\varepsilon>0$ there is a bounded, Lipschitz and $C^1$-smooth function $\varphi:M\to\Real$ such that:
\begin{enumerate}
\item $f-\varphi$ attains its strong minimum on $M$,
\item $||\varphi||_\infty <\varepsilon$ and $||d\varphi||_\infty<\varepsilon$.
\end{enumerate}
\end{cor}


\begin{rem}
It is worth noting that  if a $C^1$ Finsler manifold $M$  in the sense of Neeb-Upmeier  satisfies the \emph{DGZ smooth variational principle}, then it  is necessarily modeled on a Banach space $X$ with  a $C^1$-smooth and Lipschitz bump function. Indeed, by the DGZ smooth variational principle, there exists a $C^1$-smooth and Lipschitz function $\phi:M\to \Real$ such that $g=1-\phi$ attains its strong minimum at $x_0\in M$, $||\phi||_\infty < \frac{1}{4}$ and $||d\phi||_\infty <\frac{1}{4}$. Since $x_0$ is the strong minimum of $g$ on $M$, for every $\delta>0$ there exists $a>0$ such that $g(y)\ge a +g(x_0)$ for every $y\in M$ with $d_M(y,x_0)>\delta$ $\,(**)$. Let us take $\varphi:U\to X$ a chart with $x_0\in U$  satisfying  inequality \eqref{ineq-Neeb-Upmeier} with the constant
$K_{x_0}\ge 1$. Let us choose  $\delta>0$ such that $B_M(x_0,2\delta)\subset U$, $\varphi(B_M(x_0,2\delta))$ is bounded in $X$, and the corresponding constant $a>0$ satisfies the above condition $(**)$ for $\delta$. Let $\theta:\Real\to \Real$ be  a $C^\infty$-smooth and Lipschitz function  with $\theta(t)=1$ for every $t\le g(x_0)$, $\theta(t)=0$ for every $t\ge g(x_0) +a$ and $\Lip(\theta)\le 2/a$. Let us define $b:X\to \Real$ by $b(x)=\theta(g(\varphi^{-1}(x)))$  for every $x\in \varphi(B_M(x_0,2\delta))$, and $b(x)=0$  whenever $x\not\in  \varphi(B_M(x_0,2\delta))$. Then, it is clear that $b$ is a $C^1$-smooth bump function on
$X$ and $\Lip(b)\le \frac{K_{x_0}}{2a}$.
\end{rem}


\section{Uniformly Bumpable and Smooth Approximation}
In this section we establish a characterization of the class of separable  smooth Finsler manifolds in the sense of Neeb-Upmeier which are uniformly bumpable as those separable smooth Finsler manifolds in the sense of Neeb-Upmeier admitting approximation of Lipschitz functions by Lipschitz and smooth functions.

The notion of  \emph{smooth sup-partitions of unity} on Banach spaces was introduced by R. Fry \cite{Fry1} to solve the problem of approximation of real-valued, bounded and Lipschitz functions defined on a Banach space with separable dual
 by $C^1$-smooth and Lipschitz functions. Subsequent generalizations of this result and related results were given in \cite{azafrymon} 
  and \cite{HajekJohanis}, by means of the existence of smooth sup-partitions of unity on the Banach space. This concept can be considered  in the context of Finsler manifolds as well.


\begin{defn}
Let $M$ be a  $C^\ell$ Finsler manifold in the sense of Neeb-Upmeier. $M$ admits \textbf{$C^k$-smooth and Lipschitz sup-partitions of unity}  subordinated to an open cover  $\mathcal{U}=\{U_r\}_{r\in \Omega}$ of $M$, if there is a collection of $C^k$-smooth and $L$-Lipschitz functions $\{\psi_\alpha\}_{\alpha\in\Gamma}$ (where $L>0$ depends on $M$ and the cover $\mathcal{U}$)
such that
\begin{enumerate}
\item[(S1)] $\psi_\alpha:M\to[0,1]$ for all $\alpha\in \Gamma$,
\item[(S2)] for each $x\in M$ the set $\{\alpha\in\Gamma : \psi_\alpha(x)>0\}\in c_0(\Gamma)$,
\item[(S3)] $\{\psi_\alpha\}_{\alpha\in \Gamma}$ is subordinated to $\mathcal{U}=\{U_r\}_{r\in \Omega}$, i.e. for
 each $\alpha\in \Gamma$ there is $r\in \Omega$ such that $\supp (\psi_\alpha)\subset  U_{r}$, and
\item[(S4)] for each $x\in M$ there is $\alpha\in\Gamma$ such that $\psi_\alpha(x)=1$.
\end{enumerate}
\end{defn}


D. Azagra, R. Fry and A. Montesinos proved that  every separable Banach space with a Lipschitz and $C^k$-smooth bump function admits $C^k$-smooth sup-partitions of unity \cite{Fry1} and \cite{azafrymon}. Following their proof, it can be stated
 the existence of $C^k$-smooth sup-partitions of unity on separable, $C^k$ Finsler manifolds that are $C^k$-smooth uniformly bumpable.




\begin{thm}\label{equiv}
Let $M$ be a separable $C^\ell$ Finsler manifold  in the sense of Neeb-Upmeier. The following conditions are equivalent:
\begin{enumerate}
\item[(1)] $M$ is $C^k$-uniformly bumpable ($k\le \ell$).
\item[(2)] There is $C_2\ge 1$ (which only depends on $M$) such that for every  Lipschitz function $f:G\to\Real$ defined on an open subset of $M$ and every $\varepsilon>0$, there exists a Lipschitz and $C^k$-smooth function $g:M\to\Real$ such that $|f(x)-g(x)|<\varepsilon$ for every $x\in G$, $||dg(x)||_x \le C_2\Lip(f)$  for all $x\in M$, and thus $\Lip(g)\le C_2\Lip(f)$.
\end{enumerate}
\end{thm}
\noindent {\em Sketch of the Proof.}
The proof of  $(2)\Rightarrow (1)$ follows along the same lines as the proof of Corollary \ref{approximation:uniformly}.
The proof of the converse  is analogous to the   Banach space case   \cite{Fry1, azafrymon, HajekJohanis}
with some modifications. Let us sketch the steps of the proof for the readers convenience.

\medskip

\noindent {\em  Step 1. There is $r'>0$  such that for every $\delta<r'$ there exists a $C^k$-smooth and Lipschitz sup-partition of unity subordinated to the open cover $\{B_M(p,\delta)\}_{p\in M}$ of $M$.}
Let us mention the necessary modifications to be made in \cite{Fry1, azafrymon}
to prove this assertion.  Let us fix an equivalent $C^\infty$-smooth norm $||\cdot||$ on $c_0$, such that $||\cdot||\le ||\cdot||_\infty \le A ||\cdot||$ (for some constant
$A>1$). Since $M$ is $C^k$-uniformly bumpable, there are $r>0$ and $R>1$ such that for every point $p\in M$ and $\delta\in (0,r')$ (where $r':=\min\{r,\frac{1}{A}\}$) we can obtain  two families of $C^k$-smooth functions, $\{b_p\}_{p\in M}$  and $\{\widetilde{b}_p\}_{p\in M}$, where $b_p, \widetilde{b}_p:M\rightarrow \mathbb [0,1]$, such that
\begin{itemize}
\item[(1)] $b_p(p)=1$, $\widetilde{b}_p(p)=1$,
\item[(2)] $b_p(x)=0$ whenever $d_M(x,p)\ge \delta$, $\widetilde{b}_p(x)=0$ whenever $d_M(x,p)\ge \delta/2R$, and
\item[(3)] $\Lip(b_p)\le \sup_{x\in M} ||d b_p(x)||_x \le R/\delta$, and $\Lip(\widetilde{b}_p)\le \sup_{x\in M} ||d \widetilde{b}_p(x)||_x \le 2R^2/\delta$.
\end{itemize}
Now, by composing $\{b_p\}_{p\in M}$  and $\{\widetilde{b}_p\}_{p\in M}$ with suitable real functions, we obtain $C^k$-smooth and Lipschitz functions $f_p,g_p:M\to [0,1]$  such that $f_p(x)=0$ whenever $d_M(x,p)\le \delta/2R$, $f_p(x)=1$ whenever $d_M(x,p)\ge\delta$, $g_p(x)=1$ whenever $d_M(x,p)\le \delta/4R^2$, $g_p(x)=0$ whenever $d_M(x,p)\ge \delta/2R$,
$\Lip(f_p)\le \sup\{||d f_p(x)||_x:x\in M\}\le 3R/\delta$ and  $\Lip(g_p)\le \sup\{||dg_p(x)||_x:x\in M\}\le 6R^2/\delta$.

Now, by imitating the construction given in \cite{azafrymon, Fry1}, we obtain a  countable $C^k$-smooth and Lipschitz sup-partition of unity subordinated to $\{B_M(p,\delta)\}_{p\in M}$ with Lipschitz constant  bounded above by $15 \frac{R}{\delta} (1+2AR)$.

\medskip

\noindent {\em Step 2. There is $C_1\ge1$, that only depends on $M$, such that for every Lipschitz function $f:M\to [0,1]$ with $\Lip(f)\ge 1$, there exists a Lipschitz and $C^k$-smooth function $g:M\to \Real$ such that $|f(x)-g(x)|<1/4$ for every $x\in M$, and $\Lip(g)\le \sup\{||dg(x)||_x:x\in M\}\le C_1 \Lip(f)$.} Notice that $r':=\min\{r, \frac{1}{A}\}>0$ depends only on $M$ and $A\ge 1$ (which is an independent constant). Let us take a constant
$B > 4$, which only depends on $M$, satisfying $\frac{1}{BA}<r'$. Now, if $\Lip(f):=L\ge 1$, let us define
$\delta:=\frac{1}{BAL}< r'$. By Step 1, there exists a $C^k$-smooth and Lipschitz sup-partition of unity
$\{\varphi_n\}_{n=1}^\infty$ subordinated to $\{B_M(p,\delta)\}_{p\in M}$ such that $\Lip(\varphi_n)\le \sup\{||d\varphi_n(x)||_x:x\in M\}Ê\le 15R(1+2AR) BAL=CL$, where $C:=15R(1+2AR)BA$.  Again, by imitating the proof
of \cite{Fry1, azafrymon} , it can be checked that the function $g:M\rightarrow \mathbb R$,
\begin{equation*}
g(x):=\frac{||\{f(p_n)\varphi_n(x)\}||}{||\{\varphi_n(x)\}||} \qquad x\in M,
\end{equation*}
is $C^k$-smooth, $||dg(x)||_x\le C_1L$ for every $x\in M$, and thus it is $C_1L$-Lipschitz,  where $C_1:=2A^2C$ ($C_1$ only depends on  $M$) and   $|g(x)-f(x)|<1/4$, for every $x\in M$.






\medskip

 {\em Step 3.} Either   \cite[Proposition 1 and Theorem 3]{HajekJohanis} or \cite[Lemma 1]{Azafrykeener}
 provides the final step to ensure the existence of a constant $C_2\le 3C_1$ ($C_2$ only depends on $M$) such that every  real-valued and Lipschitz function
 $f:M\rightarrow \mathbb R$ can be uniformly approximated by a  $C^k$ smooth function $g$ such that $\Lip(g)\le \sup\{||dg(x)||_x:x\in M\}\le C_2\Lip(f)$.
 The proofs in \cite{HajekJohanis} and  \cite{Azafrykeener} are given
 for Banach spaces, but they also work for a $C^\ell$ Finsler manifold provided  Step 2 holds.


\medskip

\begin{rem}
Recall that if $M$ is a $C^{\ell}$ Finsler manifold (separable or non-separable) in the sense of Neeb-Upmeier modeled on a Banach space $X$, then condition (2) in Theorem \ref{equiv} yields to the fact that $X$ has property $(*^k)$. A proof of this fact can be obtained by applying the techniques of N. Moulis \cite{Moulis}, P. H\'ajek and M. Johanis \cite{HajekJohanis}.
Unfortunately, the results given in Section \ref{3} require additional assumptions  on the manifold $M$ to prove the converse, i. e. to prove that if  $X$ has property $(*^k)$, then $M$ satisfies condition (2) in Theorem \ref{equiv}.
\end{rem}

\section{Smooth and Lipschitz extensions}




Recall that D. Azagra, R. Fry and L. Keener proved in  \cite{Azafrykeener} that if  $X$ is a Banach space with separable dual $X^*$, there exists a constant $C>0$ (that  only depends  on $X$) such that for every closed subspace $Y\subset X$ and every $C^1$-smooth and Lipschitz function $f:Y\rightarrow \mathbb R$, there is a $C^1$-smooth and Lipschitz extension $F:X\rightarrow \mathbb R$ (i.e. $F(y)=f(y)$, for all $y\in Y$) with  $\Lip(F)\le C \Lip(f)$. Later on, a generalization of this result for  the class of Banach spaces with property $(*^1)$ was given by the authors in \cite{MarLuis}.
The above result  also holds in the case of a  $C^1$-smooth and Lipschitz function  $f:D\cap Y
\rightarrow \mathbb R$, where $D$ is a convex, closed subset of $X$, whenever there is an open  subset $U$ of $X$ with $D\subset U\subset X$ such that  $f:D\cap Y\rightarrow \mathbb R$ is Lipschitz and $C^1$-smooth (as a function on $Y$).
\begin{rem}
An examination of the constant $C$ obtained in  \cite{Azafrykeener} and \cite{MarLuis} yields to the fact that this
 can be taken as $C:= \frac{9}{4}+66C_0$, where $C_0$ is the constant given by property $(*^1)$.
 Therefore, if $X$ has property $(*^1)$ and the constant $C_0$ does not depend on the (equivalent) norm considered in $X$,
 then the constant $C$  does not depend on the (equivalent) norm considered in $X$ either.
\end{rem}
In this section we shall give a smooth extension result  on a certain class of  $C^1$ Finsler manifolds $M$ for $C^1$-smooth
functions $f:N\rightarrow \mathbb R$ defined on a submanifold $N$,
whenever $f$ is Lipschitz  with respect to the Finsler metric of the manifold $M$.

Let us consider a $C^1$ manifold $M$ modeled on a Banach space $X$.
First, let us give the definition of submanifold. A subset  $N\subset M$ is a {\em $C^1$ submanifold } of $M$ if for every $p\in N$ there is a chart
$(V_p,\varphi_p)$ of $M$ at $p$, such that $p\in V_p$  and $\varphi_p(V_p\cap N)=A\cap Y$, where $Y$ is a closed subspace of $X$ and $A$ is an open subset of $X$ with $\varphi_p(p)\in A\cap Y$.
Notice that we do not require in this definition that $Y$  is complemented in $X$. Thus, this definition of submanifold
is more general than the one considered in some texts for Banach manifolds modeled on infinite dimensional Banach spaces.
Recall that if $M$ is a $C^1$ Finsler manifold and $N$ is a $C^1$ submanifold of $M$, then $||\cdot||_{\mid_{TN}}$ is a Finsler structure for $N$ \cite[Theorem 3.6]{Palais}.
Let us begin with the following Lemma.

\begin{lem}\label{lemma:extension:approximation}
Let $X$ be a Banach space with  property $(*^1)$ and let $Y\subset X$ be a closed subspace of $X$. Let   $D \subset X$ be
a closed, convex subset of $X$, $A\subset X$ an open subset of $X$ such that $D\subset A$ and  $f:A\cap Y\to\Real$
a $C^1$-smooth and Lipschitz function (as a function on  $Y$). Let us consider $\varepsilon>0$ and a  Lipschitz extension   of $f_{\mid_{D\cap Y}}$ to $X$, which we shall denote by $F:X\to\Real$  (i.e. $F:X\to\Real$ is Lipschitz and $F(y)=f(y)$, for all $y\in D\cap Y$). Then,  there exists a $C^1$-smooth and Lipschitz function $G:X\to\Real$ such that
\begin{enumerate}
\item[(i)] $G_{\mid_{D\cap Y}}=f_{\mid_{D\cap Y}}$,
\item[(ii)] $|G(x)-F(x)|<\varepsilon$ for all $x \in X$, and
\item[(iii)] $\Lip(G)\leq R (\Lip(F)+\Lip(f))$,
\end{enumerate}
where   $R:=\frac{29}{2}C_0(\frac{9}{4}+66C_0) $  is a constant that depends only on $(X, ||\cdot||)$.
\end{lem}
\begin{proof}
Since $X$ admits property $(*^1)$, from the results in \cite{Azafrykeener} and \cite{MarLuis}, we know that there exists a Lipschitz and $C^1$-smooth extension $g:X\to\Real$ of
$f_{\mid_{D\cap Y}}$ to $X$ such that $g_{\mid_{D\cap Y}}=f_{\mid_{D\cap Y}}$ and $\Lip(g)\le C\Lip(f)$, where  $C:= \frac{9}{4}+66C_0$
depends only on $(X, ||\cdot||)$ (and $C_0$ is the constant given by property $(*^1)$). Also, since $X$ admits  property $(*^1)$, there is a $C^1$-smooth and Lipschitz function $h:X\to\Real$ such that $|h(x)-F(x)|<\varepsilon$ for $x\in X$ and $\Lip(h)\le C_0 \Lip(F)$. Consider the sets $E=\{x\in X: |g(x)-F(x)|<\varepsilon/4\}$, $I=\{x\in X: |g(x)-F(x)|\le \varepsilon/4\}$ and $B=\{x\in X: |g(x)-F(x)|<\varepsilon/2\}$  in $X$. Then $D\cap Y\subset E\subset I \subset B$.

As in the proof of \cite[Lemma 2.3]{MarLuis}, let us consider a $C^1$-smooth and Lipschitz function $u:X\to[0,1]$ such that $u(x)=1$ whenever $x\in I$, $u(x)=0$ whenever $x\in X\setminus B$ and $\Lip(u)\le \frac{9C_0(\Lip(F)+C\Lip(f))}{\varepsilon}$.

Now, let us define $G:X\rightarrow \mathbb R$,
\begin{equation*}
G(x)=u(x)g(x)+(1-u(x))h(x),  \qquad  x\in X.
\end{equation*}
 Clearly, the function $G$ is  $C^1$-smooth and $G(y)=f(y)$ for $y\in D\cap Y$. If  $x\in X\setminus B$, then
 $|G(x)-F(x)|=|h(x)-F(x)|<\varepsilon$. If $x\in B$, then we have  $|G(x)-F(x)|\le u(x)|g(x)-F(x)|+(1-u(x))|h(x)-F(x)|<\varepsilon$.
Let us prove that $G$ is Lipschitz on $X$.
\begin{itemize}
\item[(i)] If $x\in X\setminus \overline{B}$, then $||G'(x)||=||h'(x)||\le C_0 \Lip(F)$;
\item[(ii)] if $x\in \overline{B}$, then
\begin{align*}
||G'(x)||&\le ||g(x)u'(x)+h(x)(1-u)'(x)||+||g'(x)u(x)+h'(x)(1-u(x))||\le \\
&\le ||(g(x)-F(x))u'(x)+(h(x)-F(x))(1-u)'(x)||+ (C\Lip(f)+C_0\Lip(F))\le \\
&\le  \frac{29}{2}C_0C(\Lip(f)+\Lip(F)).
\end{align*}
\end{itemize}
Now, let us define $R:=\frac{29}{2}C_0C=\frac{29}{2}C_0(\frac{9}{4}+66C_0)$, which yields to $\Lip(G)\le R(\Lip(F)+\Lip(f))$.

\end{proof}

\begin{prop}\label{ext}
Let $M$ be a $C^1$ Finsler manifold  in the sense of Neeb-Upmeier $K$-weak-uniform such that it is  modeled on a Banach space $X$ that admits property $(*^1)$ and the constant $C_0$ does not depend on the norm.
Let $N\subset M$ be a closed $C^1$ submanifold and let  $f:N\to\mathbb{R}$ be a $C^1$-smooth function. If $f$ is Lipschitz  as a function on $M$ (i.e., there is $L\ge 0$ such that $|f(p)-f(q)|\le Ld_M(p,q)$, for every $p,q\in N$), then there is a $C^1$-smooth and Lipschitz extension $g:M\to\Real$ such that $\Lip(g)\leq S \Lip(f)$, where $S:=\frac{1}{2}+2RK^2$ and $R$ is the constant given in Lemma \ref{lemma:extension:approximation}.
\end{prop}
\begin{proof}
First, let us extend $f$ to $M$ as  $F(x):=\inf_{y\in N} \{f(y)+Ld_M(x,y)\}$, for every $x\in M$, where $L=\Lip(f)=
\sup\{\frac{|f(p)-f(q)|}{d_M(p,q)}:\, p,q\in N, p\not=q\}$. The function $F$ is a Lipschitz extension of $f$ to $M$, with the same Lipschitz constant  $\Lip(F)=\Lip(f)=L$.

Let us take a family of charts  $\{(O_\gamma,\varphi_\gamma)\}_{\gamma\in \Gamma_1\cup \Gamma_2}$, a set of points $\{p_\gamma\}_{\gamma \in \Gamma_1\cup \Gamma_2}$ on $M$  and a family of open subsets $\{V_\gamma\}_{\gamma \in \Gamma_1}$ on $M$ satisfying:
\begin{enumerate}
\item $\overline{V_\gamma}\subset O_\gamma$, for all $\gamma \in \Gamma_1$,
\item $p_\gamma \in O_{\gamma}$ for all $\gamma \in \Gamma:=\Gamma_1 \cup \Gamma_2$ and $p_\gamma \in V_\gamma$ for all $\gamma\in \Gamma_1$,
\item  $N\subset \bigcup_{\gamma \in \Gamma_1} V_\gamma$ and
$M\setminus N=\bigcup_{\gamma \in \Gamma_2} O_\gamma $,
\item $\varphi_\gamma(O_\gamma):=A_\gamma\subset X$ and the sets $A_\gamma$ are open  subsets of $X$, for all $\gamma \in \Gamma_1 \cup \Gamma_2$,
\item  $\varphi_\gamma(\overline{V}_\gamma):=C_\gamma\subset X$ and the sets  $C_\gamma$ are  closed, {\em convex} subsets of $X$, for all $\gamma \in \Gamma_1$,
\item $\varphi_\gamma:O_\gamma\rightarrow A_\gamma$ are $C^1$-diffeomorphisms  such that
\begin{center}
$|||d\varphi_\gamma(p)|||_\gamma:=\sup\{|||d\varphi_\gamma(p)(v)|||_\gamma: \, ||v||_p=1\}\le 2K$ {for all } $p\in O_\gamma,$
\end{center}
where $|||w|||_\gamma:=||d\varphi^{-1}_{\gamma}(\varphi_{\gamma}(p_\gamma))(w)||_{p_\gamma}$ for every $w\in X$. Moreover, we can assume that $\varphi_\gamma:O_\gamma\rightarrow A_\gamma$ is $2K$-bi-Lipschitz with the norm $|||\cdot|||_{\gamma}$ in $X$, for all $\gamma \in \Gamma_1 \cup \Gamma_2$ (notice that we are using Remark \ref{remark}(\ref{remark-R}) for $R=2K$ and Lemma \ref{desigualdades:BiLipschitz}).
\item Since  $N$ is a  $C^1$ submanifold of $M$  (modeled on a closed subspace $Y$ of $X$), we may have selected the charts so that
 $\varphi_\gamma(O_\gamma\cap N)=A_\gamma\cap Y$ and $\varphi_\gamma(\overline{V}_\gamma\cap N)=\varphi_\gamma(\overline{V}_\gamma) \cap Y=C_\gamma \cap Y$ for all $\gamma \in \Gamma_1$.
\end{enumerate}
By Lemma \ref{partition1}, there is an open refinement $\{W_{n,\gamma}\}_{n\in \mathbb N,\gamma \in \Gamma}$ of $\{V_\gamma \}_{\gamma\in \Gamma_1}\cup \{O_\gamma\}_{\gamma \in \Gamma_2}$ satisfying  properties $(i)-(iv)$ of Lemma \ref{Rudin}, and there is
a $C^1$-smooth and Lipschitz partition of unity of $M$
$\{\psi_{n,\gamma}\}_{n\in \mathbb N,\gamma\in \Gamma}$   such that   $\supp(\psi_{n,\gamma})\subset W_{n,\gamma}\subset O_\gamma$ for every $n\in\mathbb{N}$ and $\gamma\in \Gamma_2$ and  $\supp(\psi_{n,\gamma})\subset W_{n,\gamma}\subset V_\gamma$  for every $n\in\mathbb{N}$ and $\gamma\in \Gamma_1$. Let us write $L_{n,\gamma}:=
\max\{1, \sup\{||d\psi_{n,\gamma}(x)||_x:x\in M\}\}$ for every $n\in\mathbb{N}$ and $\gamma\in \Gamma$, and let us define
 $f_\gamma:A_\gamma\cap Y\to \Real$ (for all $\gamma\in \Gamma_1$) and $F_\gamma:A_\gamma\to\Real$
 (for all $\gamma\in \Gamma$) as
\begin{equation*}
f_\gamma(y):=f\circ \varphi^{-1}_\gamma(y)\ \text{ and }\ F_\gamma(x):=F\circ \varphi^{-1}_\gamma(x),
\end{equation*}
for every $y\in A_\gamma\cap Y$ and $x\in A_\gamma$.
The functions $f_\gamma$ and $F_\gamma$ are $KL$-Lipschitz with the norm
$|||\cdot|||_\gamma$ in $X$  and $F_\gamma$ is a Lipschitz extension of ${f_{\gamma}}_{\mid_{C_\gamma\cap Y}}$
to $A_\gamma$.
Since $X$ admits property $(*^1)$ (with the same constant $C_0$, for every equivalent norm), we can apply Lemma \ref{lemma:extension:approximation} to obtain  $C^1$-smooth and Lipschitz functions $G_{n,\gamma}:X\to\Real$, for all
$n \in \mathbb N$ and $\gamma \in \Gamma_1$, such that
\begin{itemize}
\item[(a)] ${G_{n,\gamma}}_{\mid_{C_\gamma\cap Y}}=f_\gamma$,
\item[(b)] $|G_{n,\gamma}(x)-F_\gamma(x)|< \frac{L}{2^{n+1} L_{n,\gamma}}$ for every $x\in A_\gamma$, and
\item[(c)] $\Lip(G_{n,\gamma})\le 2R K L$, for the norm $|||\cdot|||_\gamma$ on $X$, where $R$ is the constant given in Lemma \ref{lemma:extension:approximation}.
\end{itemize}
In addition, since $X$ has  property  $(*^1)$, we obtain for every $n \in \mathbb N$ and $\gamma\in \Gamma_2$,  $C^1$-smooth and Lipschitz functions $G_{n,\gamma}:X\to\Real$ satisfying the conditions (b) and (c) (notice that $C_0\le R$).

Now, let us define $g:M\to\Real$ by
\begin{equation*}
g(x):=\sum_{n\in \mathbb N,\gamma\in \Gamma} \psi_{n,\gamma}(x)G_{n,\gamma}(\varphi_\gamma(x)).
\end{equation*}
Since $\supp(\psi_{n,\gamma }) \subset W_{n,\gamma}\subset O_\gamma$ for every $(n,\gamma)\in   \mathbb N\times
 \Gamma$, and $\{\psi_{n,\gamma}\}_{n\in \mathbb N,
\gamma \in \Gamma}$ is a $C^1$-smooth partition of unity of $M$, the function $g$ is well defined and  $C^1$-smooth on $M$.
Now, if  $y\in N$ and  $\psi_{n,\gamma}(y)\neq 0$, then $y\in V_\gamma\cap N$ and $\varphi_\gamma(y)\in C_\gamma\cap Y$. Therefore,  $G_{n,\gamma}(\varphi_\gamma(y))=f(y)$ and $g(y)=\sum_{n\in \mathbb N,\gamma \in \Gamma}\psi_{n,\gamma}(y)f(y)=f(y)$. Let us prove that $g$ is Lipschitz on $M$. Recall that  $dg(x)=\sum_{n\in \mathbb N,\gamma \in \Gamma}d\psi_{n,\gamma}(x)=0$ for all $x\in M$, and thus  $dg(x)=\sum_{n\in \mathbb N,\gamma \in \Gamma}d\psi_{n,\gamma}(x)F(x)=0$. In addition, if  $x\in M$ and $d\psi_{n,\gamma}(x)\neq 0$, then $F(x)=F(\varphi_\gamma^{-1}(\varphi_\gamma(x)))$. Therefore,
\begin{eqnarray*}
&||dg(x)||_x&\le ||\sum_{n\in \mathbb N,\gamma \in \Gamma} G_{n,\gamma}(\varphi_\gamma(x))\,d\psi_{n,\gamma}(x)||_x
+ \\
& & \qquad \qquad \qquad \qquad \qquad  \qquad +||\sum_{n\in \mathbb N,\gamma \in \Gamma}\psi_{n,\gamma}(x)\,dG_{n,\gamma}(\varphi_\gamma(x))\,d\varphi_\gamma(x)||_x \le\\
&&\le \sum_{n\in \mathbb N,\gamma \in \Gamma}||d\psi_{n,\gamma}(x)||_x|G_{n,\gamma}(\varphi_\gamma(x))-F(x)|+
\sum_{n\in \mathbb N,\gamma \in \Gamma}\psi_{n,\gamma}(x) 2R K^2L\le \\
&& \le \sum_{n\in \mathbb N}  L_{n,\gamma(n)}\frac{L}{2^{n+1} L_{n,\gamma(n)}}+ 2R K^2 L\le\frac{ L}{2}+2RK^2L.
\end{eqnarray*}
Thus, if we   define $S:=\frac{1}{2}+2RK^2$, it follows from Proposition \ref{mean:value} that  $\Lip(g)\le \sup\{||dg(x)||_x:x\in M\}\le S\Lip(f)$ and  the constant $S$ only depends on $M$.

 \end{proof}

 Unfortunately, we do not know if the conclusion of the Proposition \ref{ext} holds if we drop the assumption that the Banach space $X$ admits the property $(*^1)$ with the same constant $C_0$ for every (equivalent) norm.




\section*{Acknowledgments}
The authors wish to thank Jes\'us Jaramillo for many helpful discussions.



\end{document}